\documentclass[a4paper,11pt]{article}
\usepackage{amsmath,amssymb,amscd,amsthm,txfonts,url}
\usepackage[backref]{hyperref}
\usepackage{pdfsync}
\usepackage[T1]{fontenc}
\usepackage[utf8]{inputenc}
\usepackage{xypic}

\newtheorem{theorem}{Theorem}[subsection]
\newtheorem{proposition}[theorem]{Proposition}
\newtheorem{corollary}[theorem]{Corollary}
\newtheorem{lemma}[theorem]{Lemma}

\numberwithin{equation}{subsection}

\theoremstyle{remark}
\newtheorem{remark}[theorem]{Remark}

\newtheorem{nothing}[theorem]{}
\theoremstyle{definition}

\def\nil{{\rm nil}}
\def\Sp{{\rm Spr}}
\def\t{{\frak{t}}}
\def\g{{\frak{g}}}
\def\b{{\frak{b}}}
\def\Ad{\mathrm{Ad}}
\def\u{{\frak{u}}}
\def\ocalT{{\overline{\calT}}}
\def\I{\mathrm{I}}
\def\R{\mathrm{R}}
\def\GL{\mathrm{GL}}

\def\calX{\mathcal{X}}
\def\calM{\mathcal{M}}
\def\calG{\mathcal{G}}
\def\calC{{\mathcal{C}}}
\def\calN{{\mathcal{N}}}
\def\car{\mathfrak{car}}
\def\calZ{\mathcal{Z}}
\def\calT{\mathcal{T}}
\def\calV{{\mathcal{V}}}
\def\calH{{\mathcal{H}}}
\def\calY{{\mathcal{Y}}}
\def\calD{{\mathcal{D}}}
\def\calU{{\mathcal{U}}}
\def\calY{{\mathcal{Y}}}
\def\calB{\mathcal{B}}

\def\calS{{\mathcal{S}}}
\def\dcb{D_\mathrm{c}^\mathrm{b}}
\def\dcbp{D_\mathrm{c}^+}
\def\dcbm{D_\mathrm{c}^-}
\def\dcbo{D_\mathrm{c}}
\def\Q{{\overline{\mathbb{Q}}_\ell}}
\def\B{\mathrm{B}}
\def\pr{\mathrm{pr}}
\def\Ind{\mathrm{Ind}}
\def\Res{\mathrm{Res}}
\def\R{\mathrm{R}}

\def\N{\mathrm{N}}

\def\p{{\mathrm{p}}}
\def\GL{\mathrm{GL}}

\def\ext{{\mathrm{Ext}}}
\def\bbA{\mathbb{A}}
\def\hom{\mathrm{Hom}}
\def\IC{{\mathrm{IC}}}
\def\p{{\mathrm{p}}}
\def\reg{{\mathrm{reg}}}
\def\rss{{\mathrm{rss}}}

\def\u{{\mathfrak{u}}}

\def\bN{{\overline{\calN}}}
\def\GL{\mathrm {GL}}

\def\B{\mathrm{B}}
\def\N{\mathrm{N}}

\def\SL{{\rm SL}}
\def\bbA{{\mathbb{A}}}
\def\sl{{\frak{sl}}}
\begin{document}

\title{On the derived Lusztig correspondence}

\author{ G\'erard Laumon
\\ {\it Universit\'e Paris-saclay, LMO CNRS UMR 8628}
\\{\tt gerard.laumon@u-psud.fr } \and Emmanuel Letellier \\ {\it
  Universit\'e Paris Cit\'e, IMJ-PRG CNRS UMR 7586} \\ {\tt
  emmanuel.letellier@imj-prg.fr} }

\pagestyle{myheadings}
\maketitle

\begin{abstract}
\noindent Let $G$ be a connected reductive group, $T$ a maximal torus of $G$, $N$ the normalizer of $T$ and $W=N/T$ the Weyl group of $G$. Let $\g$ and $\t$ be the Lie algebras of $G$ and $T$. The affine variety $\car=\t/\!/W$ of semisimple $G$-orbits of $\g$ has a natural stratification 
$$
\car=\coprod_\lambda\car_\lambda
$$
indexed  by the set of $G$-conjugacy classes of Levi subgroups : the open stratum is the set of regular semisimple orbits and the closed stratum is the set of central orbits.
\bigskip

 In \cite{Rider}, Rider considered the triangulated subcategory $\dcb([\g_\nil/G])^\Sp$ of $\dcb([\g_\nil/G])$ generated by the direct summand of the Springer sheaf. She proved that it  is equivalent to the derived category of finitely generated dg modules over the smash product algebra $\Q[W]\# H^\bullet_G(G/B)$ where $H^\bullet_G(G/B)$ is the $G$-equivariant cohomology of the flag variety. Notice that the later derived category is $\dcb(\B(N))$ where $\B(N)=[{\rm Spec}(k)/N]$ is the classifying stack of $N$-torsors.
\bigskip

The aim of this paper is to understand geometrically and generalize Rider's equivalence of categories : For each $\lambda$ we construct a cohomological correspondence inducing an equivalence of categories between $\dcb([\t_\lambda/N])$ and $\dcb([\g_\lambda/G])^\Sp$.

\end{abstract}

\tableofcontents

\maketitle

\section{Introduction}

Let $G$ a connected reductive algebraic group over an algebraic closure $k$ of a finite field. Let $T$ be a maximal torus of $G$ and $B$ be the Borel subgroup of $G$ containing $T$. We denote by $N$ the normalizer of $T$ in $G$ and we let $W=N/T$ be the Weyl group of $G$. We denote by $\g$, $\t$, and $\b$ the Lie algebras respectively of $G$, $T$ and $B$, and by $\g_\nil$ the nilpotent cone of $\g$.
\bigskip

In \cite{Rider}, Laura Rider considered the triangulated subcategory $\dcb([\g_\nil/G])^\Sp$ of $\dcb([\g_\nil/G])$ generated by the direct summand of the Springer sheaf. She proved that it  is equivalent to the derived category of finitely generated dg modules over the smash product algebra $\Q[W]\# H^\bullet_G(G/B)$ where $H^\bullet_G(G/B)$ is the $G$-equivariant cohomology of the flag variety (see also \cite{Li} for a similar result).
\bigskip

Denote by $\B(N)=[{\rm Spec}(k)/N]$ the classifying stack of $N$-torsors. Rider's result can be then reformulated as an equivalence of categories between $\dcb(\B(N))$ and $\dcb([\g_\nil/G])^\Sp$. The aim of this paper is to construct  this equivalence via a cohomological correspondence between $\B(N)$ and $[\g_\nil/G]$ in the spirit of Lusztig induction. Namely we construct  a complex $\bN_\nil$ on $\B(N)\times[\g_\nil/G]$ such that the functor
$$
\dcb(\B(N))\rightarrow\dcb([\g_\nil/G])^\Sp,\hspace{.5cm}K\mapsto \pr_{2*}\,\underline{\rm Hom}\left(\bN_\nil,\pr_1^!(K)\right)
$$
is an equivalence of categories, where $\pr_1$ and $\pr_2$ are the two obvious projections.
\bigskip

Let us remark that $\B(N)$ and $[\g_\nil/G]$ are the zero fibers of the canonical maps 

$$
[\t/N]\rightarrow\car,\hspace{1cm}[\g/G]\rightarrow\car
$$
where 
$$
\car:=\t/\!/W=\mathrm{Spec}(k[\t]^W).
$$
In this paper we consider more generally the analogous equivalences above  the strata of the natural stratification 
$$
\car=\coprod_{\lambda\in\frak{L}}\car_\lambda
$$
where $\frak{L}$ is the set of $G$-conjugacy classes of Levi subgroups (of parabolic subgroups) of $G$. 

%For a stack $\calX$ over $\car$ we put $\calX_\lambda=\calX\times_\car\car_\lambda$. 
\bigskip

\bigskip

 We consider the quotient stacks
$$
\calG=[\g/G],\hspace{0.5cm}\calB=[\b/B],\hspace{.5cm}\calT=[\t/T],\hspace{0.5cm}\ocalT=[\t/N]
$$
with respect to the adjoint  actions ${\rm Ad}$ (which is trivial for $T$) and we denote by $\pi:\calT\rightarrow\ocalT$ the induced map. Notice that our staks  are algebraic stacks over $\car$.

We consider the following commutative diagram
\begin{equation}
\xymatrix{&&\calB\ar[rrd]^p\ar[lld]_q\ar[d]^{(q,p)}&&\\
\calT\ar[d]^-\pi&&\calT\times_\car\calG\ar[rr]^-{\pr_\calG}\ar[ll]_{\pr_\calT}\ar[d]_-{\pi\times 1}&&\calG\\
\ocalT&&\ocalT\times_\car\calG\ar[urr]_-{\overline{\pr}_\calG}\ar[ll]_-{\overline{\pr}_\calT}&&}
\label{diagintro}\end{equation}
where $q$ is induced by the projection $\b\rightarrow\t$ and $p$ is induced by the inclusion $\b\subset\g$.

Put 

$$
\calN:=(q,p)_!\Q.
$$

The Lusztig induction and restriction functors are defined by
\begin{align*}
&\Ind:\dcb(\calT)\rightarrow\dcb(\calG),\,\,\,\,\,K\mapsto p_*q^!K=\pr_{\calG\, *}\underline{\rm Hom}\left(\calN,\pr_\calT^!(K)\right)\\
&\\
&\Res:\dcb(\calG)\rightarrow\dcb(\calT),\,\,\,\,\, K\mapsto q_!p^*K=\pr_{\calT\, !}\left(\calN\otimes\pr_\calG^*(K)\right)
\end{align*}

For a stack $\calX$ over $\car$ we put 

$$
\calX_\lambda=\calX\times_\car\car_\lambda.
$$
We prove the following result (see Theorem \ref{maintheo-descent}).

\begin{theorem} [Descent] For $\lambda\in\frak{L}$, the restriction $\calN_\lambda$ of $\calN$ to $(\calT\times_\car\calG)_\lambda$ descends to $\bN_\lambda$ on $(\ocalT\times_\car\calG)_\lambda$.
\label{mainTHEO}\end{theorem}

To prove the existence of $\bN_\lambda$, we regard the complex $(q_\lambda,p_\lambda)_!\Q$ (where $(q_\lambda,p_\lambda)$ is obtained from $(q,p)$ by base change) as the outcome of a \emph{Postnikov diagram} which descends in a natural way to a Postnikov diagram on $(\ocalT\times_\car\calG)_\lambda$ using weight arguments. The complex $\bN_\lambda$ is then defined as the outcome of the descended Postnikov diagram.
\vspace{.5cm}

\begin{remark}(1) As a particular case, the restriction of $\calN$ to $\B(T)\times[\g_\nil/G]$ descends to a complex $\bN_\nil$ on $\B(N)\times[\g_\nil/G]$.

\noindent (2) As noticed by S. Gunningham \cite{Gun1}, the functor $\Res$ depends on the choice of the Borel subgroup containing $T$ and so the restriction of a complex $K\in\dcb(\calG)$ can not be $W$-equivariant.  Therefore we can not expect that $\calN$ descends to $\ocalT\times_\car\calG$.  In \ref{calcul} we use an old computation of Verdier which explain more directly why $\calN$ can not descend. Therefore the above descent result seems to be optimal.

\noindent (3) We also prove that $\calN$ descends over regular elements, which stratum intersects with all strata above $\car$ (see \S \ref{descentrege}).

\end{remark}

For each $\lambda\in\frak{L}$ we define the pair of adjoint functors $(\R_\lambda,\I_\lambda)$ by

\begin{align*}
&\R_\lambda:\dcb(\calG_\lambda)\rightarrow\dcb(\ocalT_\lambda),\hspace{.5cm}K\mapsto\pr_{1!}\left(\bN_\lambda\otimes \pr_2^*(K)\right),\\
&\I_\lambda:\dcb(\ocalT_\lambda)\rightarrow\dcb(\calG_\lambda),\hspace{.5cm}K\mapsto \pr_{2*}\,\underline{\rm Hom}\left(\bN_\lambda,\pr_1^!(K)\right)
\end{align*}

If $\Ind_\lambda$ and $\Res_\lambda$ denote the induction and restriction above $\car_\lambda$ and $\pi_\lambda:\calT_\lambda\rightarrow\ocalT_\lambda$ the morphism obtained by base change from $\pi$ then

$$
\Res_\lambda=\pi_\lambda^*\circ\R_\lambda,\hspace{.5cm}\Ind_\lambda=\I_\lambda\circ\pi_{\lambda\,!}.
$$

For each geometric point $c$ of $\car_\lambda$ we have  the functor
$$
\xymatrix{\Ind_c:\dcb(\calT_c)\ar[r]&\dcb(\calG_c),}
$$
which is compatible by base change with the functors $\Ind_\lambda$.
\bigskip

For each geometric point $c$ of $\car$, define $\dcb(\calG_c)^{\rm {Spr}}$ as the triangulated subcategory of $\dcb(\calG_c)$ generated by the direct factors of $\Ind_c(\Q)$. We then define $\dcb(\calG_\lambda)^{\rm Spr}$ as the full subcategory of $\dcb(\calG_\lambda)$ of complexes $K$ such that $K_c\in\dcb(\calG_c)^{\rm Spr}$ for all geometric point $c$ of $\car_\lambda$.
\bigskip

We prove the following theorem (see Theorem \ref{adj2-lambda} and Remark \ref{rem-lambda}).

\bigskip

\begin{theorem}
The functor $\I_\lambda$ induces an equivalence of categories $\dcb(\ocalT_\lambda)\rightarrow\dcb(\calG_\lambda)^{\rm Spr}$ with inverse given by $\R_\lambda$.

If $G$ is of type $A$ with connected center, then $\dcb(\calG_\lambda)^{\rm Spr}=\dcb(\calG_\lambda)$ and so  $\I_\lambda:\dcb(\ocalT_\lambda)\rightarrow\dcb(\calG_\lambda)$ is an equivalence of categories with inverse functor $\R_\lambda$.
\label{mainTHEO2}\end{theorem}

\begin{remark}
(1) Notice that the stack $\overline{\calT_c}$ with $c=0$ is $\B(N)=[\mathrm{Spec}(k)/N]$. It follows from Theorem \ref{mainTHEO2} that the category $\dcb(\B(N))$ is equivalent to $\dcb(\calG_0)^{\rm Spr}$. 

\noindent (2) When $\lambda$ is the stratum of semisimple regular elements,  the two stacks $\ocalT_\lambda$ and $\calG_\lambda$ are isomorphic and the functors $\I_\lambda$, $\R_\lambda$ are the identity functors (see \S \ref{semisimplereg}).

\noindent (3) The interested reader may follow the methods of \cite[\S 6]{BBD} to get the similar results over the complex numbers.
\end{remark}
\bigskip

Consider the projection 

$$
\xymatrix{\calT\simeq\B(T)\times\t\ar[r]^-s&\t.}
$$
The functor $s^![{\rm dim}\, T]({\rm dim}\, T)$ is an equivalence $\calM(\t)\rightarrow\calM(\calT)$ between the categories of perverse sheaves. Therefore, when working with perverse sheaves, we may work with the diagram

$$
\xymatrix{\calT\ar[d]_{s}&&\calB\ar[ll]_q\ar[rrd]^p\ar[dll]_{q'}\ar[d]^{(q',p)}&&\\
\t&&\t\times_\car\calG\ar[ll]^-{\pr_1}\ar[rr]_-{\pr_2}&&\calG}
$$
which is more convenient as the kernel $(q',p)_!\Q$ is the intersection cohomology complex of $\t\times_\car\calG$ and so it descends naturally to $[\t/W]\times_\car\calG$. Moreover, by Bezrukavnikov and Yom Din \cite{BY}, the functor $\Ind$ and $\Res$ maps perverse sheaves to perverse sheaves. 
\bigskip

The above factorization of $p$ via $(q',p)$ and its generalization in the framework of the generalized Springer correspondence is due to Lusztig (see \cite[Proof of Proposition 5.5.3]{Letellier} for more details  and \cite{LuNC} for the case where $G$ is not necessarily connected).

\bigskip

In \S \ref{perInd}, we define a pair of adjoint functors ${^p}\I:\calM([\t/W])\rightarrow\calM(\calG)$ and ${^p}\R:\calM(\calG)\rightarrow\calM([\t/W])$ and we prove the following theorem (see Theorem \ref{maintheo-perverse}).

\begin{theorem}The functor ${^p}\I$ induces an equivalence of categories $\calM([\t/W])\rightarrow\calM(\calG)^{\rm Spr}$ with inverse functor ${^p}\R$. In particular if $G$ is of type $A$ with connected center, then $\calM(\calG)^{\rm Spr}=\calM(\calG)$ and so the categories $\calM([\t/W])$ and $\calM(\calG)$ are equivalent.
\end{theorem}

This result is an analogue in the $\ell$-adic setting of (a special case of)  the main result of  S. Gunningham \cite{Gun} in the  $D$-module setting. We were informed by the anonymous referee that using the result of Bezrukavnikov and Yom Din, the proof of Gunningham works also in the $\ell$-adic setting (although this is not published). 

\bigskip

\noindent {\bf Acknowledgments :} It is a pleasure to thank P. Achar, O. Dudas, S. Gunningham, B. Hennion, L. Illusie, B. Keller, and O. Schiffmann for useful discussions, O. Gabber for telling  us about Verdier's paper "Un calcul triste" and M. Romagny for his help with group action on stacks.

\section{Preliminaries}

\subsection{Postnikov diagrams in triangulated categories}

Let $\calD$ be a triangulated category.
\bigskip

We will need the following lemma.

\begin{lemma}
Consider a diagram in $\calD$
$$
\xymatrix{B[-1]\ar[r]^-u\ar[d]^{\beta[-1]}&C\ar[r]^-v&A\ar[r]^-w\ar[d]^\alpha&B\ar[d]^\beta\\
B'[-1]\ar[r]^-{u'}&C'\ar[r]^-{v'}&A'\ar[r]^-{w'}&B'}
$$
where $(u,v,w)$ and $(u',v',w')$ are distinguished triangles.

Suppose that
$$
\mathrm{Hom}(B,A')=0,\hspace{.5cm}\text{ and }\hspace{.5cm}\mathrm{Hom}(A,B'[-1])=0.
$$
Then there exists a unique morphism $C\rightarrow C'$ which extends the above diagram into a morphism of distinguished triangles.
\bigskip
\label{lem0}\end{lemma}

\begin{proof}
Only the unicity needs to be proved (the existence follows from the axioms of a triangulated category). We are thus reduced to prove that if we have a morphism of triangles
$$
\xymatrix{B[-1]\ar[r]^-u\ar[d]^0&C\ar[r]^-v\ar[d]^{\beta}&A\ar[d]^0\ar[r]^-w&B\ar[d]^0\\
B'[-1]\ar[r]^-{u'}&C'\ar[r]^-{v'}&A'\ar[r]^-{w'}&B'}
$$
then $\beta=0$.
\bigskip

As $\beta\circ u=0$, there exists $\gamma:A\rightarrow C'$ such that $\gamma\circ v=\beta$. Since
$$
v'\circ\gamma\circ v=0
$$
there exists $\varepsilon:B\rightarrow A'$ such that $v'\circ\gamma=\varepsilon\circ w$. By assumption, $\varepsilon=0$ and so $v'\circ\gamma=0$.

There exists thus $\overline{\gamma}:A\rightarrow B'[-1]$ such that $u'\circ\overline{\gamma}=\gamma$. By assumption we must have $\overline{\gamma}=0$ and so $\gamma=0$ from which we get that $\beta=0$.
\end{proof}

\noindent  A Postnikov diagram $\Lambda$ \cite{Orlov} consists of a complex 

$$
\xymatrix{A_m\ar[r]^{\partial_m}&A_{m-1}\ar[r]^{\partial_{m-1}}&\cdots\ar[r]^{\partial_2}&A_1\ar[r]^{\partial_1}&A_0}
$$
in $\calD$, called the base and denoted by $\Lambda_b$, together with  a finite sequence of distinguished triangles
$$
\xymatrix{C_i\ar[r]^{\alpha_i}&A_i\ar[r]^-{d_i}&C_{i-1}\ar[r]^-+&},
$$
with $i=1,\dots,m$, such that $\partial_i=\alpha_{i-1}\circ d_i$ for all $i$. We visualize $\Lambda$ as
\begin{equation}
\xymatrix{C_m\ar[d]_{\alpha_m}&&C_{m-1}\ar[ll]_+\ar[d]_{\alpha_{m-1}}&&\cdots\ar[ll]_+&&C_2\ar[ll]_+\ar[d]_{\alpha_2}&&C_1\ar[ll]_+\ar[d]_{\alpha_1}&&C_0\ar[ll]_+\ar@{=}[d]\\
A_m\ar[rr]^{\partial_m}\ar[rru]^{d_m}&&A_{m-1}\ar[rr]^{\partial_{m-1}}\ar[rru]^{d_{m-1}}&&\cdots\ar[rr]&&A_2\ar[rr]^{\partial_2}\ar[rru]^{d_2}&&A_1\ar[rr]^{\partial_1}\ar[rru]^{d_1}&&A_0}
\label{DIAGRAM}
\end{equation}
The \emph{length} of $\Lambda$ is the integer $m$.
\bigskip

The object $C_m$ will be called the \emph{outcome} of the Postnikov diagram $\Lambda$.

\bigskip

Notice that if $\calD$ is equipped with a non-degenerate $t$-structure, then we can define the Postnikov diagram $\Lambda(K)$ with outcome $K$  for any complex $K\in\calD^{[n,n+m]}$ as 
\begin{footnotesize}
$$
\xymatrix{K\ar[d]_{\alpha_m}&(\tau_{\leq n+m-1}K)[1]\ar[l]_-+\ar[d]_{\alpha_{m-1}}&(\tau_{\leq n+m-2}K)[2]\ar[l]_-+\ar[d]_{\alpha_{m-2}}&\cdots\ar[l]_-+\\
\calH^{n+m}(K)[-n-m]\ar[r]^-{\partial_m}\ar[ru]^{d_m}&\calH^{n+m-1}(K)[-n-m+2]\ar[r]^-{\partial_{m-1}}\ar[ru]^{d_{m-1}}&\calH^{n+m-2}(K)[-n-m+4]\ar[r]&\cdots}
$$
\end{footnotesize}

\begin{footnotesize}
\begin{equation}
\xymatrix{\cdots&(\tau_{\leq n+1}K)[m-1]\ar[l]_-+\ar[d]_{\alpha_1}&(\tau_{\leq n}K)[m]\ar[l]_-+\ar@{=}[d]\\
\cdots\ar[r]^-{\partial_2}&\calH^{n+1}(K)[m-n-2]\ar[r]^-{\partial_1}\ar[ru]^{d_1}&\calH^{n}(K)[m-n]}
\label{DIAG}\end{equation}
\end{footnotesize}
using the distinguished triangles
\begin{footnotesize}
$$
\xymatrix{\tau_{\leq n-m+2i-1}\left(\tau_{\leq n+i}(K)[m-i]\right)\ar[r]&\tau_{\leq n+i}(K)[m-i]\ar[r]&\tau_{\geq n-m+2i}\left(\tau_{\leq n+i}(K)[m-i]\right)=\calH^{n+i}(K)[m-n-2i]\ar[r]^-+&}
$$
\end{footnotesize}
The construction of $\Lambda(K)$ is functorial in $K$.
\bigskip

\begin{remark} Given $K\in\calD^{[n,n+m]}$, then $\Lambda(K)$ is the unique (up to a unique isomorphism) Postnikov diagram of the form (\ref{DIAGRAM}) with outcome $K$ such that 

 \begin{equation}
C_i\in\calD^{\leq n-m+2i}\hspace{1cm}\text{and}\hspace{1cm} A_i\in \calD^{\geq n-m+2i}
\label{CA}\end{equation}
for all $i=1,\dots,m$. This follows from \cite[Proposition 1.3.3(ii)]{BBD}.
\label{remPostnikov}\end{remark}

\begin{lemma}Given $K\in \calD^{[n,n+m]}$, the Postnikov diagram $\Lambda(K)$ is the unique one (up a unique isomorphism) that complete the subdiagram

\begin{footnotesize}
$$
\xymatrix{K\ar[d]_{\alpha_m}&&&&\\
\calH^{n+m}(K)[-n-m]\ar[r]^-{\partial_m}&\calH^{n+m-1}(K)[-n-m+2]\ar[r]^-{\partial_{m-1}}&\cdots\ar[r]^-{\partial_1}&\calH^n(K)[m-n]}
$$
\end{footnotesize}

\label{Lambda}

\end{lemma}

\begin{proof}By Remark \ref{remPostnikov}, it is enough to prove that if we are given a Postnikov diagram of the form (\ref{DIAGRAM}) where $A_i=P_i[m-n-2i]$ with $P_i$ a perverse sheaf, then $C_i\in\calD^{[i-(m-n),2i-(m-n)]}$ for all $i=1,\dots,m$. The proof goes by recurrence on $m\geq 1$ using the long exact sequence of perverse cohomology.
\end{proof}

\bigskip

\begin{lemma}
Assume that we have a Postnikov diagram with the notation as in diagram (\ref{DIAGRAM}) and that for each $i=0,\dots,m-2$, any $j\geq i+2$ and any $k\geq 1$ we have
$$
\mathrm{Hom}(A_j,A_i[-k])=0.
$$
Then if $(d'_m,d'_{m-1},\dots,d'_1)$ is a sequence of arrows $d'_i:A_i\rightarrow C_{i-1}$ such that $\partial_i=\alpha_{i-1}\circ d'_i$ for all $i$, then
$$
(d'_m,\dots,d'_1)=(d_m,\dots,d_1).
$$
\label{DIAGLEM}\end{lemma}

\begin{proof}
We need to prove that
$$
\mathrm{Hom}(A_{i+2},C_i[-1])=0
$$
for all $i=0,\dots,m-2$.
\bigskip

Since we have a distinguished triangle
$$
\xymatrix{C_{i-1}[-2]\ar[r]& C_i[-1]\ar[r]&A_i[-1]\ar[r]^-+&}
$$
we have an exact sequence
$$
\xymatrix{\mathrm{Hom}(A_{i+2},C_{i-1}[-2])\ar[r]&\mathrm{Hom}(A_{i+2},C_i[-1])\ar[r]&\mathrm{Hom}(A_{i+2},A_i[-1]).}
$$
By assumption the right hand side is $0$ and so we are reduced to prove that
$$
\mathrm{Hom}(A_{i+2},C_{i-1}[-2])=0.
$$
Repeating the argument as many time as needed we end up with proving that $\mathrm{Hom}(A_{i+2},C_0[-1-i])=0$ which follows also from the assumption as $C_0=A_0$.
\end{proof}

\begin{proposition}
Assume given a complex $\Lambda_b$ in $\calD$
\begin{equation}
\xymatrix{A_m\ar[r]&A_{m-1}\ar[r]&\cdots\ar[r]&A_1\ar[r]&A_0}
\label{ladderdescent1}
\end{equation}
and suppose that it satisfies
\begin{equation}
\mathrm{Hom}(A_j,A_i)=0,\hspace{.5cm}\text{ for all } j<i,
\label{COND11}
\end{equation}
and
\begin{equation}
\mathrm{Hom}(A_j,A_i[-k])=0,\hspace{.5cm}\text{ for all }j>i\text{ and } k\geq 1.
\label{COND21}
\end{equation}
Then the complex (\ref{ladderdescent1}) can be completed in a unique way (up to a unique isomorphism) into a Postnikov diagram $\Lambda$ in $\calD$.
\label{laddercompletion}\end{proposition}

\begin{proof}
The proof goes by induction on $0\leq r <m$. If $r=0$, this is obvious. Suppose that we have proved the proposition up to the rank $r$. Therefore we have the complex $C_r$. By Lemma \ref{DIAGLEM} there exists a unique map $d_{r+1}$ such that the following diagram commutes
$$
\xymatrix{&C_r\ar[d]^{\alpha_r}&C_{r-1}\ar[d]^{\alpha_{r-1}}\ar[l]&\ar[l]\cdots\\
A_{r+1}\ar[ru]^{d_{r+1}}\ar[r]^{\partial_{r+1}}&A_{r}\ar[r]^{\partial_r}\ar[ru]^{d_r}&A_{r-1}\ar[r]\ar[ru]&\cdots}.
$$
We complete $d_{r+1}$ into a distinguished triangle
$$
\xymatrix{C_r[-1]\ar[r]&C_{r+1}\ar[r]&A_{r+1}\ar[r]^{d_{r+1}}& C_r}.
$$
Now we notice that
$$
\mathrm{Hom}(A_{r+1},C_r[-1])=0,\hspace{.5cm}\text{ and }\hspace{.5cm}\mathrm{Hom}(C_r,A_{r+1})=0.
$$
Indeed, the complexes $C_r$ being successive extensions of the complexes $A_i$, with $i=0,\dots,r$, this follows from the assumptions (\ref{COND11}) and (\ref{COND21}) .

By Lemma \ref{lem0}, such a $C_{r+1}$ is thus unique (up to a unique isomorphism).
\end{proof}

\subsection{Stacks and sheaves : notation and convention}

Let $k$ be an algebraic closure of a finite field. In these notes, unless specified, our stacks are $k$-algebraic stacks of finite type and quotient stacks are denoted by $[X/G]$ if $G$ is an agebraic group over $k$ acting a $k$-scheme $X$. If $X={\rm Spec}(k)$ we put $\B(G)=[X/G]$.
\bigskip

 For a morphism $f:\calX\rightarrow\calY$ of stacks, we have the usual functors 
 
 $$
 f_*=Rf_*:\dcbp(\calX)\rightarrow\dcbp(\calY),\hspace{1cm} f_!=Rf_!:\dcbm(\calX)\rightarrow\dcbm(\calY),$$
$$
 f^*, f^!:\dcbo^\bullet(\calY)\rightarrow\dcbo^\bullet(\calX),\hspace{.5cm}(\text{with } \bullet\in\{\emptyset,-,+,b\})
 $$
 between derived categories of constructible $\Q$-sheaves (see \cite{LO2}).  
 
 \begin{remark} The fact that $f^!$ preserves $-,+,b$ follows from from our assumption that our stacks are of finite type over a field. Indeed, the statement reduces to the case where $\calX$ and $\calY$ are schemes of finite type over a field which is known  \cite[Corollary 2.9]{SGA}. 
  \end{remark}

 We will use freely the properties of these functors (adjunction, projection formula, base change,...). When there is no ambiguity we will sometimes write $K|_\calX$ instead of $f^*(K)$.
 \bigskip

 \begin{remark} (i) Notice that if $d\in\mathbb{Z}$ is  the dimension of a fiber of $f$ whose absolute value is maximal, then
 
 $$
 f_!:\dcbo^{]-\infty,n]}(\calX)\rightarrow \dcbo^{]-\infty,n+2d]}(\calY).
 $$
 For instance  
 $$
 H_c^i([X/G],\Q)=0\hspace{.5cm} \text{ if }\hspace{.5cm}i> 2({\rm dim}\, X-{\rm dim}\, G).
 $$

\noindent  (ii) If $f$ is representable  then $f_!$  induces
 $$
 f_!:\dcb(\calX)\rightarrow\dcb(\calY).
 $$
 
 \noindent (iii) If $f$ is smooth with fibers of pure relative dimension $d$, then \cite[9.1.2]{LO2}
 
 $$
 f^!=f^*[2d](d).
 $$
 \label{f!}\end{remark}

\bigskip

Except in some rare occasions, we will only need the category $\dcb$.

\bigskip

\noindent We will denote by $\overline{\mathbb{Q}}_{\ell,\, \calX}$ the constant sheaf on $\calX$. If there is no ambiguities, we will sometimes denote it simply by $\Q$ to alleviate the notation.
\bigskip

For two $\calY$-stacks $\calX$ and $\calX'$, we define the external tensor product of $K\in\dcbm(\calX)$ and $K'\in\dcbm(\calX')$ above $\calY$ as
$$
K\boxtimes_\calY K':=\pr_1^*K\otimes\pr_2^*K'\in\dcbm(\calX\times_\calY\calX').
$$
where $\pr_1$ and $\pr_2$ are the two projections.
\bigskip

We consider the auto-dual perversity $p$ and we denote by $\calM(\calX)$ the full subcategory of $\dcb(\calX)$ of perverse sheaves on $\calX$. Then if $\calX$ is an equidimensional stack with smooth dense open substack $\calU$, the intersection cohomology complex on $\calX$ with coefficient in a local system $\mathcal{E}$ on $U$ is denoted by $\IC_\calX(\mathcal{E})$ ; its restriction to $\calU$ is $\mathcal{E}$. If $\mathcal{E}=\Q$ we will simply write $\IC_\calX$ instead of $\IC_\calX(\Q)$. Recall that $\IC_\calX(\mathcal{E})[{\rm dim}\,\calX]\in\calM(\calX)$ is  the image of
$$
{^p}\calH^0(j_!\mathcal{E}[{\rm dim}\,\calX])\rightarrow {^p}\calH^0(j_*\mathcal{E}[{\rm dim}\,\calX]),
$$
where $j:\calU\rightarrow\calX$ is the inclusion. Recall also that if $\calX'\rightarrow\calX$ is a small resolution of singularities (representable, proper, birational, $\calX'$ smooth), then $f_*\Q=\IC_{\calX}$.

 If $D_\calX$ denotes the Verdier dual, then $D_\calX(\IC_\calX(\mathcal{E}))=\IC_\calX(\mathcal{E}^\vee)[2\mathrm{dim}(\calX)](\mathrm{dim}(\calX))$ where $\mathcal{E}^\vee$ denotes the dual local system.
\bigskip

\begin{proposition}
Let $\calX$ be an equidimensional algebraic stack and let $\calU$ be a dense open smooth substack of $\calX$. Suppose that we have a short exact sequence
$$
\xymatrix{0\ar[r]&A'\ar[r]&A\ar[r]&A''\ar[r]&0}
$$
in the category of perverse sheaves on $\calX$. If both $A'$ and $A''$ are the intermediate extension of their restriction to $\calU$, then $A$ is also the intermediate extension of its restriction to $\calU$.
\label{IE}\end{proposition}

\begin{proof}
Let $j$ be the inclusion of $\calU$ in $\calX$. From $j_!j^*\rightarrow 1\rightarrow j_*j^*$ and the fact that $A$ is a perverse sheaf we have a commutative diagram
$$
\xymatrix{&A\ar[rd]&\\
{^p}\calH^0j_!j^*A\ar[ru]\ar@{->>}[rd]&&{^p}\calH^0j_*j^*A\\
&j_{!*}j^*A\ar@{^{(}->}[ru]&}
$$
from which we can identify $j_{!*}j^*A$ as a subquotient of $A$. In particular
$$
\mathrm{length}(j_{!*}j^*A)\leq \mathrm{length}(A).
$$
It thus enough to prove that
\begin{equation}
\mathrm{length}(j_{!*}j^*A)\geq \mathrm{length}(A).
\label{ineq}\end{equation}

We have the commutative diagram of perverse sheaves
$$
\xymatrix{&{^p}\calH^0j_!j^*A'\ar[r]\ar@{->>}[d]&{^p}\calH^0j_!j^*A\ar@{->>}[d]\ar[r]&{^p}\calH^0j_!j^*A''\ar[r]\ar@{->>}[d]&0\\
0\ar[r]&A'=j_{!*}j^*A'\ar@{^{(}->}[d]\ar[r]&j_{!*}j^*A\ar[r]\ar@{^{(}->}[d]&A''=j_{!*}j^*A''\ar[r]\ar@{^{(}->}[d]&0\\
0\ar[r]&{^p}\calH^0j_*j^*A'\ar[r]&{^p}\calH^0j_*j^*A\ar[r]&{^p}\calH^0j_*j^*A''&}
$$
where the top and bottom horizontal sequences are exact. We thus deduce that the middle sequence is exact at $A'$ and $A''$ and the composition $A'\rightarrow A''$ is zero, from which we deduce the inequality (\ref{ineq}) as
$$
\mathrm{length}(A)=\mathrm{length}(A')+\mathrm{length}(A'').
$$
\end{proof}

As we assume that $k$ is the algebraic closure of a finite field and that all  our stacks are of finite type, any stack we consider in this paper will be defined over some finite subfield of $k$.
\bigskip

If $\calX$ is defined over finite subfield $k_o\subset k$, we denote by $\calX_o$ the corresponding $k_o$-structure on $\calX$ (if no confusion arises) and we denote by $\mathrm{Frob}_o:\calX\rightarrow\calX$ the induced geometric Frobenius on $\calX$.
\bigskip

Let $K,K'\in\dcbo^b(\calX)$. We recall that $H_c^i(\calX,K)$ (resp. $\mathrm{Ext}^i(K,K')=\mathrm{Hom}(K,K'[i])$) is said to be pure of weight $r$ if there exists a finite subfield $k_o$ of $k$ such that $\calX$ and $K$ (resp. $K$ and $K'$) are defined over $k_o$ and the eigenvalues of the induced automorphism $\mathrm{Frob}_o^*$ on $H_c^i(\calX,K)$ (resp. $\mathrm{Ext}^i(K,K')$) are algebraic numbers whose complex conjugates are all of absolute value $|k_o|^{r/2}$. It is well-known that this notion of weight does not depend on the choice of $k_o$. In particular the Tate twist $\Q(i)$ makes sense over $k$. 
\bigskip

Let $\calX_o$ be a $k_o$-structure on $\calX$ and $K_o, K'_o\in\dcbo^b(\calX_o)$, and let us denote by $K$ and $K'$ the induced complexes in $\dcb(\calX)$. Then the following sequence \cite[(5.1.2.5)]{BBD}
\begin{equation}
\xymatrix{0\ar[r]&\mathrm{Ext}^{i-1}(K,K')_{\mathrm{Frob}_o}\ar[r]&\mathrm{Ext}^i(K_o,K'_o)\ar[r]&\mathrm{Ext}^i(K,K')^{\mathrm{Frob}_o}\ar[r]&0}
\label{BBD}\end{equation}
is exact for all $i$.
\bigskip

Notice that if $\mathrm{Ext}^j(K,K')=0$ for odd values of $j$, then it follows from the above exact sequence that for all $i$
$$
\mathrm{Ext}^{2i}(K_o,K'_o)\rightarrow\mathrm{Ext}^{2i}(K,K')^{\mathrm{Frob}_o}.
$$
is an isomorphism.

\subsection{Cohomology of $\B(T)$}

Let $T$ be a rank $d$ torus over $k$. Recall that for any $k$-scheme $S$, the category $\B(T)(S)$ is the category of $T$-torsors over $S$ and the algebraic stack $\B(T)$ is of dimension $-d$.
\bigskip

The cohomology $H^\bullet(\B(T),\Q)$ is concentrated in non-negative even degrees  and the morphism

$$
c_{\rm can}:X^*(T)\rightarrow H^2(\B(T),\Q)(1)
$$
given by Chern classes induces an isomorphism

$$
X^*(T)\otimes \Q(-1)\simeq H^2(\B(T),\Q).
$$
As $\Q$-algebras, $H^{2\bullet}(\B(T),\Q)$ is isomorphic to ${\rm Sym}^\bullet\big(X^*(T)\otimes \Q(-1)\big)$.
\bigskip

By duality,  $H_c^\bullet(\B(T),\Q)$ is thus concentrated in even degrees $\leq -2\,{\rm dim}(T)$ and $H_c^{2i}(\B(T),\Q)$ is pure of weight $2i$.

\subsection{$W$-equivariant complexes}\label{W-equivariant}

Let $W$ be a finite group and a stack $\calX$. An action of $W$ on a complex $K\in\dcb(\calX)$ is a group homomorphism

$$
\theta: W\longrightarrow {\rm Aut}(K).
$$ 

In this case we define the $W$-invariant part $K^W\rightarrow K$ in $\dcb(\calX)$ of $(K,\theta)$ as follows.
\bigskip

Notice first that in any additive category, if for some morphism $e:A\rightarrow A$, there exists two morphisms $u:A\rightarrow A'$ and $v:A'\rightarrow A$ such that $v\circ u=e$ and $u\circ v=1_{K'}$ (in which case $e$ is an idempotent which is said to be \emph{split}), then $1-e$ admits a kernel which is $v$ and a cokernel which is $u$. In particular, $(A',u,v)$ is unique up to a unique isomorphism and we call it the \emph{splitting} of $e$.
\bigskip

Since $\dcb(\calX)$ is a triangulated category with bounded $t$-structure, by the main theorem of \cite{Le}, every idempotent element $e$ of $\mathrm{End}(K)$ splits. Considering the indempotent
\begin{equation}
e=e_K:=\frac{1}{|W|}\sum_{w\in W}\theta(w)\in\mathrm{End}(K),
\label{idempotent}\end{equation}
we define $K^W\rightarrow K$ as the kernel of $1-e$.
\bigskip

A \emph{$W$-torsor} is a morphism of stacks $\pi:\calX\rightarrow\overline{\calX}$ that fits to a cartesian diagram

\begin{equation}
\xymatrix{\calX\ar[r]\ar[d]_\pi&{\rm Spec}(k)\ar[d]^{\pi_o}\\
\overline{\calX}\ar[r]&\B(W)}
\label{BW}\end{equation}
\begin{remark} (1) A $W$-torsor $\pi:\calX\rightarrow\overline{\calX}$ is thus finite \'etale representable and for any scheme $S$ and morphism $S\rightarrow\calY$, the projection

$$
S\times_{\overline{\calX}}\calX\rightarrow S
$$
has a natural structure of $W$-torsor (between schemes).

\noindent (2) If $\pi:\calX\rightarrow\overline{\calX}$ is a $W$-torsor then $\calX$ is equipped with a right action of $W$ in the sense of \cite{Romagny} and conversely from right action of $W$ on $\calX$ we get a $W$-torsor by taking the quotient morphism of $\calX$ by $W$. However, although this is implicit, we will not use the definition of group actions on stacks nor the notion of quotient of stacks by group action.
\end{remark}
\bigskip

An $\ell$-adic sheaf on $\B(W)$ is a vector space equipped with a right action of $W$ and 

$$
(\pi_o)_*\Q=\Q[W]
$$
for right multiplication of $W$ on itself.  The left multiplication corresponds to the Galois action. We thus have a decomposition

$$
(\pi_o)_*\Q=\bigoplus_\chi V_\chi\otimes \mathcal{L}_{o,\chi}
$$
where the sum is over the irreducible $\Q$-characters of $W$, $V_\chi$ is a $W$-module affording the character $\chi$ and $\mathcal{L}_{o,\chi}$ is the irreducible smooth $\ell$-adic sheaf on $\B(W)$ corresponding to $V_\chi^*$.

By base change from Diagram (\ref{BW}) we get the analogous decomposition 

\begin{equation}
\pi_*\Q=\bigoplus_\chi V_\chi\otimes \mathcal{L}_\chi
\label{decompofor}\end{equation}
Notice that $\mathcal{L}_\chi$ may not be irreducible and that $(\pi_*\Q)^W=\Q$.
\bigskip

Given a $W$-torsor $\pi:\calX\rightarrow\overline{\calX}$, we denote by $\dcb(\calX,W)$ the subcategory of $\dcb(\calX)$ whose objects are isomorphic to objects of the form $\pi^*(\overline{K})$ with $\overline{K}\in\dcb(\overline{\calX})$ and morphisms are given by 

$$
{\rm Hom}_{\dcb(\calX,W)}(\pi^*\overline{A},\pi^*\overline{B}):=\pi^*{\rm Hom}(\overline{A},\overline{B})\simeq {\rm Hom}(\overline{A},\overline{B}).
$$
We call $\dcb(\calX,W)$ the category of \emph{$W$-equivariant complexes} on $\calX$ (with respect to $\pi$). 
\bigskip

\begin{remark}\label{rem-inverse}The inverse functor of $\pi^*:\dcb(\overline{\calX})\rightarrow\dcb(\calX,W)$ is given by 

$$
\dcb(\calX,W)\rightarrow\dcb(\overline{\calX}),\hspace{.5cm}K\mapsto (\pi_*K)^W.
$$
Indeed this follows from the above discussion and 

$$
\pi_*\pi^*\overline{K}=(\pi_*\Q)\otimes\overline{K}
$$
which is a consequence of the projection formula.
\end{remark}

\bigskip

Given two $W$-equivariant complexes $A=\pi^*\overline{A}$ and $B=\pi^*\overline{B}$ in $\dcb(\calX)$, we define an action of $W$ on ${\rm Hom}_{\dcb(\calX)}(A,B)$ as follows. 

We have

\begin{align*}
{\rm Hom}_{\dcb(\calX)}(A,B)&={\rm Hom}_{\dcb(\calX)}(\pi^*\overline{A},\pi^*\overline{B})\\
&={\rm Hom}_{\dcb(\overline{\calX})}(\overline{A},\pi_*\pi^*\overline{B})\\
&= \bigoplus_\chi V_\chi\otimes{\rm Hom}_{\dcb(\overline{\calX})}\left(\overline{A}, \mathcal{L}_\chi\otimes\overline{B}\right).
\end{align*}

The action of $W$ on the $V_\chi$ defines thus an action of $W$ on ${\rm Hom}_{\dcb(\calX)}(A,B)$ and we have

$$
{\rm Hom}_{\dcb(\calX)}(A,B)^W={\rm Hom}_{\dcb(\calX,W)}(A,B).
$$
\bigskip

\begin{remark}Our definition of $W$-equivariant complexes is consistent with the usual one. If $W$ acts (on the right) on a scheme $X$, then following \cite[III, 15]{KW}, a $W$-equivariant complex on $X$ is a pair $(K,\theta)$ with $K\in\dcb(X)$ and $\theta=(\theta_w)_{w\in W}$ a collection of isomorphisms
$$
\theta_w:w^*(K)\rightarrow K
$$
such that

(i) $\theta_{ww'}=\theta_w\circ w^*(\theta_{w'})$ for all $w,w'\in W$, and

(ii) $\theta_1=1_K$,

\noindent where $1_K:K\rightarrow K$ denotes the identity morphism. 

Moreover, if $(K,\theta),(K',\theta')\in \dcb(X,W)$, then the group $W$ acts (on the left) on ${\rm Hom}_{\dcb(X)}(K,K')$ as 

$$
w\cdot f=\theta'_w\circ w^*(f)\circ(\theta_w)^{-1}
$$
for all $w\in W$ and $f\in {\rm Hom}(K,K')$, and
$$
{\rm Hom}_{\dcb(X,W)}(K,K'):={\rm Hom}_{\dcb(X)}(K,K')^W.
$$
\end{remark}
\bigskip

\noindent A commutative diagram

$$
\xymatrix{\calX\ar[rr]^f\ar[d]^\pi&&\calX'\ar[d]^{\pi'}\\
\overline{\calX}\ar[rr]&&\overline{\calX}'}
$$
where $\pi$ and $\pi'$ are $W$-torsors, induces a functor  $f^*:\dcb(\calX',W)\rightarrow\dcb(\calX,W)$,  and  if moreover the above diagram is cartesian and $f$ is representable  we also get a funtor $f_!:\dcb(\calX,W)\rightarrow\dcb(\calX',W)$. Both are  compatible with the usual functors $f^*$ and $f_!$ where we forget the actions of $W$.
\bigskip

\begin{remark} Notice that if $K$ is $W$-equivariant, then its Postnikov diagram $\Lambda(K)$ is  $W$-equivariant, i.e., the vertices and the arrows of the diagram are $W$-equivariant. 
\label{W-Hom}\end{remark}
\bigskip

In the next two following lemmas we assume given a $W$-torsor

$$
\pi:\calX\rightarrow\overline{\calX}.
$$

\begin{lemma}
Assume that we have a distinguished triangle
$$
\xymatrix{C'\ar[r]^f&\pi^*\overline{A}\ar[r]^d&\pi^*\overline{C}\ar[r]^-h&C'[1]\ar[r]&}
$$
in $\dcb(\calX)$ and that $d$ is $W$-equivariant. Then there exists a distinguished triangle
\begin{equation}
\xymatrix{\overline{C}'\ar[r]^{\overline{f}}&\overline{A}\ar[r]^{\overline{d}}&\overline{C}\ar[r]^-{\overline{h}}&\overline{C}'[1]}
\label{DT}\end{equation}
in $\dcb(\overline{\calX})$ and an isomorphism of triangles
$$
\xymatrix{C'\ar[rr]^f\ar[d]^-s&&A\ar[rr]^d\ar@{=}[d]&&C\ar[rr]^h\ar@{=}[d]&&C'[1]\ar[d]^-{s[1]}\\
\pi^*(\overline{C}')\ar[rr]^{\pi^*(\overline{f})}&&\pi^*(\overline{A})\ar[rr]^{\pi^*(\overline{d})}&&\pi^*(\overline{C})\ar[rr]^{\pi^*(\overline{h})}&&\pi^*(\overline{C}')[1]}
$$

If moreover we assume that
$$
\mathrm{Hom}(A,C[-1])=0,\hspace{.5cm}\text{ and }\hspace{.5cm}\mathrm{Hom}(C,A)=0
$$
then the triangle (\ref{DT}) is unique (up to a unique isomorphism) and the morphism $s$ is unique.
\label{DIAGLEM2}\end{lemma}

\begin{proof}
As $d:\pi^*\overline{A}\rightarrow\pi^*\overline{C}$ is $W$-equivariant, by definition it descends to a unique morphism $\overline{d}:\overline{A}\rightarrow\overline{C}$ in $\dcb(\overline{\calX})$. We complete this morphism into a distinguished triangle
$$
\xymatrix{\overline{C}'\ar[r]^{\overline{f}}&\overline{A}\ar[r]^{\overline{d}}&\overline{C}\ar[r]^-{\overline{h}}&\overline{C}'[1]}
$$
in $\dcb(\overline{\calX})$.
\bigskip

\noindent There exists an isomorphism $s:C'\rightarrow\pi^*(\overline{C}')$ such that the following diagram commutes
$$
\xymatrix{C'\ar[rr]^f\ar[d]^-s&&A\ar[rr]^d\ar@{=}[d]&&C\ar[rr]^h\ar@{=}[d]&&C'[1]\ar[d]^-{s[1]}\\
\pi^*(\overline{C}')\ar[rr]^{\pi^*(\overline{f})}&&\pi^*(\overline{A})\ar[rr]^{\pi^*(\overline{d})}&&\pi^*(\overline{C})\ar[rr]^{\pi^*(\overline{h})}&&\pi^*(\overline{C}')[1]}
$$
The second statement follows from Lemma \ref{lem0}.
\end{proof}
\bigskip

\begin{proposition}
Assume that we have a Postnikov diagram $\Lambda$ of the form (\ref{DIAGRAM}) in $\dcb(\calX)$ such that the complex $\Lambda_b$ is the image by $\pi^*$ of a complex
\begin{equation}
\xymatrix{\overline{A}_m\ar[r]&\overline{A}_{m-1}\ar[r]&\cdots\ar[r]&\overline{A}_1\ar[r]&\overline{A}_0}
\label{ladderdescent}\end{equation}
in $\dcb(\overline{\calX})$. Suppose also that
\begin{equation}
\mathrm{Hom}(A_j,A_i)=0,\hspace{.5cm}\text{ for all } j<i,
\label{COND1}\end{equation}
and
\begin{equation}
\mathrm{Hom}(A_j,A_i[-k])=0,\hspace{.5cm}\text{ for all }j>i\text{ and } k\geq 1.
\label{COND2}\end{equation}
Then the complex (\ref{ladderdescent}) can be completed in a unique way (up to a unique isomorphism) into a Postnikov diagram $\overline{\Lambda}$ in $\dcb(\overline{\calX})$ such that
$$
\pi^*(\overline{\Lambda})\simeq\Lambda.
$$
\label{W-prop}\end{proposition}

\begin{proof}
The conditions (\ref{COND1}) and (\ref{COND2}) imply the analogous conditions with $A_i$ replaced by $\overline{A}_i$. Therefore by Lemma \ref{laddercompletion} we can extend in a unique way the complex (\ref{ladderdescent}) into a Postnikov diagram $\overline{\Lambda}$. We prove by induction on $0\leq r <m$ that $\pi^*(\overline{\Lambda})=\Lambda$ using Lemma \ref{DIAGLEM2} and the unicity of $d_{r+1}$ (Lemma \ref{DIAGLEM}).
\end{proof}

\subsection{Cohomological correspondences}\label{cor}
\bigskip

By a cohomological correspondence $\Gamma=(\calC,N,p,q)$ we shall mean a correspondence of $S$-algebraic stacks
\begin{equation}
\xymatrix{&\calC\ar[dr]^p\ar[dl]_q&\\
\calY&&\calX}
\label{model}\end{equation}
together with a kernel $N\in\dcbm(\calC)$.
\bigskip

The correspondence $\Gamma$ comes with a functor
$$
\Res=\Res_\Gamma:\dcbm(\calX)\rightarrow\dcbm(\calY),\hspace{.5cm}A\mapsto q_!(N\otimes p^*A)
$$
which we call the \emph{restriction functor} associated to $\Gamma$, whose right adjoint is the \emph{induction functor} 

$$
\Ind=\Ind_\Gamma:\dcbp(\calY)\rightarrow\dcbp(\calX),\hspace{.5cm}B\mapsto p_*\underline{\rm Hom}(N,q^!B).
$$

\begin{lemma}
If we have a morphism of correspondences $f:(\calC,p,q)\rightarrow (\calC',p',q')$, i.e. a commutative diagram
\begin{equation}
\xymatrix{&\calC\ar@/^/[ddr]^p\ar@/_/[ddl]_q\ar[d]^f&\\
&\calC'\ar[dr]_{p'}\ar[dl]^{q'}&\\
\calY&&\calX}
\label{morcordiag}\end{equation}
then we have natural isomorphisms
$$\Res_{(\calC,N,p,q)}=\Res_{(\calC',f_! N,p',q')},\hspace{1cm}\Ind_{(\calC,N,p,q)}=\Ind_{(\calC',f_!N,p',q')}.$$
\label{morcor}\end{lemma}

\begin{proof}
This is an obvious consequence of the projection formulas (see \cite[9.1.1 and 9.1.i]{LO2}).
\end{proof}
\bigskip

\begin{lemma}[Composition]
Assume that we have two cohomological correspondences $\Gamma=(\calC,N,p,q)$ and $\Gamma'=(\calC',N',p',q')$. Consider the following diagram

$$
\begin{footnotesize}
\xymatrix{&&\calZ\times_S\calX\ar@/_3pc/[dddll]_{\overline{\pr}_1}\ar@/^3pc/[dddrr]^{\overline{\pr}_2}&&\\
&&\calC'\times_\calY\calC\ar[rd]^{\pr_2}\ar[ld]_{\pr_1}\ar[u]_{(q',p)}&&\\
&\calC'\ar[rd]^{p'}\ar[ld]_{q'}&&\calC\ar[rd]^{p}\ar[ld]_{q} &\\
\calZ&&\calY&&\calX}
\end{footnotesize}$$
Then
$$
\Res_{\Gamma'}\circ\Res_\Gamma=\Res_{\Gamma'\circ\Gamma}
$$
where
$$
\Gamma'\circ\Gamma:=\left(\calZ\times_S\calX,(q',p)_!(N'\boxtimes_\calY N),\overline{\pr}_2,\overline{\pr}_1\right).
$$
\label{comp}\end{lemma}

We keep the correspondence (\ref{model}) and we assume that it can be completed into a diagram

$$
\xymatrix{\calY\ar[d]^-\pi&&\calC\ar[rr]^p\ar[ll]_q\ar[d]^-\rho&&\calX\\
\overline{\calY}&&\overline{\calC}\ar[rru]_{\overline{p}}\ar[ll]^{\overline{q}}&&
}
$$
where $\pi$ and $\rho$ are $W$-torsors and where the square is cartesian.

If $N$ is $W$-equivariant, i.e. descends to a complex $\overline{N}$ on  $\overline{\calC}$. Then we have the following factorization

\begin{equation}
\Ind=\I\circ\pi_*,\hspace{1cm}\Res=\pi^*\circ\R
\label{fact}\end{equation}
where $\I:\dcbp(\overline{\calY})\rightarrow\dcbp(\calX)$ and $\R:\dcbm(\calX)\rightarrow\dcbm(\overline{\calY})$ are  respectively the induction and restriction functors defined from the correspondence $(\overline{\calC},\overline{N},\overline{p},\overline{q})$.
\bigskip

\begin{remark}\label{Inv}Notice that $\I$ can be computed as

$$
\I(\overline{K})={\rm Ind}(\pi^*\overline{K})^W
$$
since

\begin{align*}
\Ind(\pi^*\overline{K})&=\I\circ\pi_*\circ\pi^*(\overline{K})\\
&=\bigoplus_\chi V_\chi\otimes \I(\mathcal{L}_\chi\otimes \overline{K}),
\end{align*}
where $\pi_*\Q=\bigoplus_\chi V_\chi\otimes \mathcal{L}_\chi$ is the decomposition  indexed by the irreducible characters of $W$ (see Remark \ref{rem-inverse}). 
\end{remark}

\subsection{Reductive groups}\label{reductive}
\bigskip

For an affine connected algebraic group $H$ over $k$, we denote by $\frak{h}$ the Lie algebra of $H$. We denote by $Z(H)$ the center of $H$ and by $z(\frak{h})$ the center of $\frak{h}$. We will also denote by $\calH:=[\frak{h}/H]$ the quotient stack of $\frak{h}$ by $H$ for the adjoint action $\Ad:H\rightarrow\GL(\frak{h})$.

If moreover $H$ is reductive we denote by
$$
\frak{car}_\frak{h}=\frak{h}/\!/H:=\mathrm{Spec}(k[\frak{h}]^H)
$$
the variety of characteristic polynomials of the elements of $\frak{h}$ and we have a canonical map $\chi_\frak{h}:\mathcal{H}\rightarrow\car_\frak{h}$.

If $H$ is commutative, then $\calH\simeq\frak{h}\times\B(H)$, $\car_\frak{h}\simeq\frak{h}$ and $\chi_\frak{h}$ is the first projection.
\bigskip

From now, $G$ is a connected reductive group over $k$, $T$ is a maximal torus of $G$, $B$ a Borel subgroup of $G$ containing $T$, $U$ the unipotent radical of $B$, $N$ the normalizer $\N_G(T)$ of $T$ in $G$ and $W$ the Weyl group $N/T$. Through this paper we put
\bigskip
$$
n:=\mathrm{dim}\, T.
$$
\bigskip

We will simply use the notation $\car$ instead of $\car_\g$ and we recall that
$$
\frak{car}=\t/\!/W.
$$
We let
$$
\pi:\calT\longrightarrow\ocalT=[\t/N]
$$
be the canonical map.
\bigskip

The character group is denoted by $X^*(T)$ and the cocharacter group by $X_*(T)$.
\bigskip

When it makes sense, we use freely the subscripts $\reg$ and $\rss$ for restriction to $G$-regular or $G$-regular semisimple elements. Similarly we will use the subscript ${\rm nil}$ for restriction to nilpotent elements.
\bigskip

We assume throughout this paper that the  characteristic of $k$
is not too small so that $\t_{\reg}\neq\emptyset$  and centralizers of semisimple elements of $\g$ are Levi subgroups (of parabolic subgroups) of $G$ (see \cite[\S 2.6]{Letellier} for an explicit bound on $p$).
\bigskip

\subsection{Lusztig correspondence}

Consider the correspondence (which we call Lusztig correspondence)

\begin{equation}
\xymatrix{\calT&\calB\ar[l]_q\ar[r]^p&\calG}
\label{Ind}\end{equation}
where the arrows are induced by the inclusion $\b\subset\g$ and by the projection $\b\rightarrow\t$.
\bigskip

Consider

$$
X:=\{(x,gB)\in\g\times G/B\,|\, \Ad(g)(x)\in\b\}
$$
The group $G$ acts on $X$  by $g\cdot(x,hB)=(
\Ad(g)(x),ghB)$ and the map 

$$B\rightarrow X,\,\, x\mapsto(x,B)$$
induces an isomorphism $\calB\rightarrow\calX:=[X/G]$.
\bigskip

Under the identification $\calB\simeq\calX$, the morphism $p$ is the quotient by $G$ of the Grothendieck-Springer resolution

$$
\pr_1:X\rightarrow \g,\,\, (x,gB)\mapsto x
$$
from which we deduce the first assertion of the theorem below.

\begin{theorem} (i) The morphism $p:\calB\rightarrow\calG$  is representable, proper and semi-small.
\bigskip

\noindent (ii) We have a factorization of $p$  (Stein factorization) 
$$
\xymatrix{\calB\ar[rr]^-{(q',p)}&&\calS:=\t\times_\car\calG\ar[rr]^{\pr_2}\ar[d]\ar@{}[drr] | {\square}&&\calG\ar[d]\\
&&\t\ar[rr]&&\car}
$$
where $q'=\chi_\t\circ q:\calB\rightarrow\calT\rightarrow\t$, and $(q',p)$ is a small resolution of singularities.
\label{Steinpro}\end{theorem}

\begin{proof}
Above $\calS_{\rss}:=\t\times_\car\calG_{\rss}$, $(q',p)$ is an isomorphism and so $\calG_{\rss}$ is smooth. The complementary of $\calG_{\rss}$ in $\calS$ is of codimension at least $2$ therefore $\calS$ satisfies the condition $(R_1)$ of Serres's criterion for normality. Moreover $\t$ is a complete intersection over $\car$ and so by base change the same is true for $\calS$ over $\calG$. As $\calG$ is smooth, the stack $\calS$ satisfies the condition $(S_2)$ of Serres's criterion. Therefore $\calS$ is normal from which we deduce the proposition.
\end{proof}

\subsection{Lusztig induction and restriction}\label{paraind}

The morphism $p$ is representable, proper and $q$ is smooth of pure dimension $0$. We thus have $p_*=p_!$ and $q^*=q^!$.
The Lusztig induction and restriction functors \cite[(7.1.7)]{Lu} are the induction and restriction defined from the correspondence (\ref{Ind}) with the constant sheaf as a kernel, i.e. 
$$
\Ind:\dcb(\calT)\rightarrow\dcb(\calG),\hspace{.5cm}K\mapsto p_*q^!(K),
$$
$$
\Res:\dcb(\calG)\rightarrow\dcb(\calT),\hspace{.5cm}K\mapsto q_!p^*(K).
$$
\begin{remark}Notice that $\Res$ is well-defined as $q_!$ maps bounded complexes to bounded complexes. To see that we consider the commutative triangle

$$
\xymatrix{\calT&&\calB\ar[ll]_q\\
&&[\b/T]\ar[u]_h\ar[llu]^a}
$$
Since $h$ is an affine fibration, the  adjunction morphism

$$
h_!h^*\rightarrow 1
$$
is an isomorphism and so we conclude by noticing that $a$ is representable.

\end{remark}

\bigskip

By the main result of \cite{BY}, the functors $\Ind$  and $\Res$ maps perverse sheaves to perverse sheaves.

Consider the factorization
\begin{equation}
\xymatrix{&&\calB\ar[dll]_q\ar[drr]^p\ar[d]^{(q,p)}&&\\
\calT&&\calT\times_{\frak{car}}\calG\ar[ll]_{\pr_\calT}\ar[rr]^{\pr_\calG}&&\calG}
\label{Indbis}\end{equation}
where the map $\calT\rightarrow\car$ is either the composition $\calT\rightarrow\calG\rightarrow\car$ or the composition $\calT\rightarrow\t\rightarrow\car$.
\bigskip

By \S\ref{cor}, we have

\begin{align*}
\Res(K)&={\pr_\calT}_!\left(\pr^*_\calG(K)\otimes (q,p)_!\Q\right),\\
&\\
\Ind(K)&={\pr_\calG}_*\,\underline{\rm Hom}\left((q,p)_!\Q,\pr_\calT^!(K)\right)\\
&={\pr_\calG}_!\left(\pr^*_\calT(K)\otimes (q,p)_!\Q\right).
\end{align*}
The last identity follows from the fact that $q^!=q^*$ and $p_!=p_*$ since we can then regard $\Ind$ as the restriction functor associated to the correspondence $(\calB,q,p)$.
\bigskip

Notice also that $(q,p)_!\Q\in\dcb(\calT\times_\car\calG)$.

\subsection{Induction and restriction for perverse sheaves}\label{perInd}

As we say in the previous section, the two functors $\Ind$ and $\Res$ are $t$-exact but, as we will see later, the kernel $(q,p)_!\Q$ is not $W$-equivariant. In this section we introduce slightly modified functors ${^p}\underline{\Ind}$ and ${^p}\underline{\Res}$ defined from a kernel which is naturally $W$-equivariant.
\bigskip

We consider the following commutative diagram

\begin{equation}
\xymatrix{\calT\ar[d]_{s}&&\calB\ar[ll]_q\ar[rrd]^p\ar[dll]_{q'}\ar[d]^{(q',p)}&&\\
\t&&\calS=\t\times_\car\calG\ar[ll]^-{\pr_1}\ar[rr]_-{\pr_2}&&\calG}
\label{DIAGPER}\end{equation}
where $s:\calT\simeq\t\times\B(T)\rightarrow\t$ is the projection on $\t$. Notice that $s$ is not representable.
\bigskip

Consider the induction functor 

$$
\underline{\Ind}:=\Ind\circ s^![n](n):\dcbp(\t)\rightarrow\dcbp(\calG),\,\,\,\,K\mapsto p_*q'{^!}(K)[n](n).
$$
Since the functor  $s^![n](n):\calM(\t)\rightarrow\calM(\calT)$ is an equivalence of categories with inverse functor ${^p}\calH^0\circ \left(s_![-n](-n)\right)$ and since $\Ind$ maps perverse sheaves to perverse sheaves (see \S \ref{paraind}), the functor $\underline{\Ind}$ preserves perversity and we denote by ${^p}\underline{\Ind}$ the induced functor on perverse sheaves, i.e. 
$$
{^p}\underline{\Ind}:=\Ind\circ s^![n](n):\calM(\t)\rightarrow\calM(\calG),\,\,\,\,K\mapsto p_*q'{^!}(K)[n](n).
$$
The functor $\underline{\Ind}$ admits a left adjoint 

$$
\underline{\Res}:=s_!\circ\Res[-n](-n):\dcbm(\calG)\rightarrow\dcbm(\t),\,\,\, K\mapsto q'_!p^*(K)[-n](-n)
$$
which  is right $t$-exact but not left $t$-exact (unlike $\underline{\Ind}$). 
\bigskip

Since $\Res$ preserves perversity,  for $K\in\calM(\calG)$ we have

$$
\underline{\Res}(K)=s_!s^!(K')=K'\otimes R\Gamma_c(\B(T),\Q)[-2n](-n)
$$
where $K'\in\calM(\t)$ is defined as $s^!(K')[n](n)=\Res(K)$.
\bigskip

Define the restriction

$$
{^p}\underline{\Res}:={^p}\calH^0\circ\underline{\Res}:\calM(\calG)\rightarrow\calM(\t),\,\,\,K\mapsto K'.
$$

\begin{lemma} ${^p}\underline{\Res}$ is left adjoint to ${^p}\underline{\Ind}$.
\end{lemma}

\begin{proof}We have 

\begin{align*}
{\rm Hom}_{\calM(\t)}\left({^p}\underline{\Res}(A),B\right)&={\rm Hom}_{\calM(\calT)}\left(s^!({^p}\underline{\Res}(A))[n](n),s^!B[n](n)\right)\\
&={\rm Hom}_{\calM(\calT)}\left(\Res(A),s^!B[n](n)\right)\\
&={\rm Hom}_{\calM(\calG)}\left(A,\Ind(s^!B[n](n))\right)\\
&={\rm Hom}_{\calM(\calG)}\left(A,{^p}\underline{\Ind}(B)\right).
\end{align*}

\end{proof}
 
 Put 
 
 $$
 \overline{\t}:=[\t/W]
 $$
and let $\pi_\t:\t\rightarrow\overline{\t}$ be the quotient map.
\bigskip

\begin{proposition} The functors ${^p}\underline{\Ind}$ and ${^p}\underline{\Res}$ factorize as 

$$
{^p}\underline{\Ind}={^p}\I\circ\pi_{\t\, *},\hspace{1cm}{^p}\underline{\Res}=\pi_\t^*\circ{^p}\R
$$
where ${^p}\I:\calM(\overline{\t})\rightarrow\calM(\calG)$ and ${^p}\R:\calM(\calG)\rightarrow\calM(\overline{\t})$ are defined as 
\begin{align*}
&{^p}\I(K):=\pr_{2\, *}\,\underline{\rm Hom}\left(\IC_{\overline{\t}\times_\car\calG},\pr_1^!K\right)[n](n),\\
&{^p}\R(K):={^p}\calH^0\left(\pr_{1\, !}\left(\IC_{\overline{\t}\times_\car\calG}\otimes\pr_2^*(K)\right)[-n](-n)\right).
\end{align*}
\label{propdecomp}
\end{proposition}

\begin{proof}Follows from the diagram (\ref{DIAGPER}) using that $(q',p)_!\Q=\IC_{\t\times_\car\calG}$ which follows from the fact that $(q',p)$ is a small resolution of singularities by Theorem \ref{Steinpro} (ii).

\end{proof}

We have the following proposition.

\begin{proposition}

(1) Let $K\in\calM(\t)$ be $W$-equivariant  and let $L\in \calM(\calG)$. The adjunction isomorphism

$$
{\rm Hom}_{\calM(\t)}\big({^p}\underline{\Res}(L),K\big)\simeq {\rm Hom}_{\calM(\calG)}\big(L,{^p}\underline{\Ind}(K)\big)
$$
is $W$-equivariant.

(2) The functor ${^p}\R$ is left adjoint to ${^p}\I$.

\end{proposition}

\begin{proof}From Proposition \ref{propdecomp} we find that

$$
{^p}\underline{\Ind}(w\cdot f)=w\cdot{^p}\underline{\Ind}(f)
$$
for any $f\in{\rm Hom}({^p}\underline{\Res}(L),K)$ and $w\in W$, from which we deduce the assertion (1). The second assertion is a consequence of (1) together with following identities

\begin{align*}
&{\rm Hom}_{\calM(\overline{\t})}\left({^p}\R(L),\overline{K}\right)={\rm Hom}_{\calM(\t)}\left({^p}\underline{\Res}(L),K\right)^W,\\
&{\rm Hom}_{\calM(\calG)}\left(L,{^p}\I(\overline{K})\right)={\rm Hom}_{\calM(\calG)}\left(L,{^p}\underline{\Ind}(K)\right)^W
\end{align*}
where  $K=\pi_\t^*(\overline{K})$.
\end{proof}

\subsection{Springer action}

From Borho-MacPherson construction of the Springer action \cite{BM}, the complex $p_!\Q$ is endowed with an action of the Weyl group $W$ (i.e. a group homomorphism $W\rightarrow {\rm Aut}(p_!(\Q))$). Their strategy was to prove that $p_*\Q$ is the intermediate extension of some smooth $\ell$-adic sheaf on $\calG_\rss$ on which $W$ acts. From the diagram (\ref{DIAGPER}) we see that $p_!\Q=\pr_{2\, !}\IC_{\t\times_\car\calG}$ and so the action of $W$ follows from the fact that $\IC_{\t\times_\car\calG}$ is naturally $W$-equivariant.
\bigskip

This defines an action of $W$ on $H^i(\calB,\Q)$ as well as an action on the cohomology of the fibres. These actions are compatible with the restriction maps from the cohomology of $\calB$
 to the cohomology of the fibers.
 \bigskip
 
 On the other hand the natural map $\B(T)\rightarrow \B(B)$ is a $U$-fibration and so $H^i(\B(T),\Q)=H^i(\B(B),\Q)$. The action of $W$ on $\B(T)$  induces thus an action of $W$ on $H^i(\B(B),\Q)$. This action coincides with the Springer action (regarding $\B(B)$  as the zero fibre of $p$).
 \bigskip
 
 \begin{lemma}The restriction map $H^i(\B(B),\Q)\rightarrow H^i(\calB,\Q)$ induced by the  map $\calB\rightarrow \B(B)$ (given by the $B$-torsor $\b\rightarrow\calB$) is $W$-equivariant for the Springer actions.
 \label{canW} \end{lemma}
 
 \begin{proof}We know that the restriction map
$$
i^*:H^i(\calB,\Q)\rightarrow H^i(\B(B),\Q)
$$
of the inclusion $i:\B(B)\hookrightarrow \calB$ is $W$-equivariant for the Springer actions. 
\bigskip

The morphism $b:\calB\rightarrow\B(B)$ is a vector bundle with fiber $\b$. Hence $b^*$ is an isomorphism with inverse $i^*$ and so is $W$-equivariant.
\end{proof}
\bigskip

\section{Steinberg stacks}\label{SS}

\subsection{Geometry of Steinberg stacks}

We consider the following stacks

$$
\calZ:=\calB\times_\calG\calB,\hspace{1cm}\calY:=\calB\times_\calS\calB.
$$
We have the following cartesian diagrams

$$
\xymatrix{\calY\ar[rr]\ar[d]&&\calZ\ar[d]\\
\calS\ar[d]\ar[rr]&&\calS\times_\calG\calS\ar[d]\\
\t\ar[rr]&&\t\times_\car\t}
$$
where the top horizontal arrow is the natural map and the last two ones are the diagonal morphisms. Since the diagonal morphism $\t\rightarrow\t\times_\car\t$ is closed, the stack $\calY$ is a closed substack of $\calZ$.
\bigskip

\begin{remark}Consider
$$
Z=\{(x,gB,g'B)\in \g\times G/B\times G/B\,|\, \Ad(g^{-1})(x)\in \b,\, \Ad({g'}^{-1})(x)\in \b\}
$$
so that $\calZ=[Z/G]$ where $G$ acts on $Z$ by
$$
(x,gB,g'B)\cdot h=(\Ad(h^{-1})x,h^{-1}gB,h^{-1}g'B).
$$
\end{remark}
\bigskip

The canonical morphisms $\calB\rightarrow\B(B)$ and $\calG\rightarrow\B(G)$ induce a morphism
$$
\xymatrix{\calZ\ar[r]^-\varphi&\B(B)\times_{\B(G)}\B(B)}
$$
and the canonical map $[B\backslash G/B]\rightarrow\B(B)\times_{\B(G)}\B(B)$ is an isomorphism.
\bigskip

Consider the stratification
$$
[B\backslash G/B]=\coprod_{w\in W}\mathcal{O}_w,
$$
where $\mathcal{O}_w=[B\backslash BwB/B]$. Notice that the natural map $\B(B_w)\rightarrow\mathcal{O}_w$ is an isomorphism.

This induces partitions into locally closed substacks

$$
\calZ=\coprod_{w\in W}\calZ_w,\hspace{1cm}\calY=\coprod_{w\in W}\calY_w.
$$

For $w\in W$, put
$$
B_w:=B\cap wBw^{-1},\hspace{.5cm}\b_w:={\rm Lie}(B_w),\hspace{.5cm}U_w:=U\cap wUw^{-1}, \hspace{.5cm}\u_w:={\rm Lie}(U_w).
$$

Then

$$
\calZ_w=[\b_w/B_w],\hspace{1cm}\calY_w=[(\t^w+\u_w)/B_w]
$$
where $B_w$ acts by the adjoint action and $\t^w$ are the points of $\t$ fixed by $w$.

\bigskip

Since $B_w=TU_w$, we proved the following proposition.

\begin{proposition}
(1) For all $w\in W$ the projection
$$
\calZ_w\rightarrow\calT
$$
induced by the canonical projection $\b\rightarrow \t$ is a fibration with fibre isomorphic to $ [\u_w/U_w]$ where $U_w$ acts by the adjoint action.
\bigskip

(2) The fibers of the projection $\calZ_w\rightarrow\mathcal{O}_w$ (resp. $\calY_w\rightarrow\mathcal{O}_w$) are affine spaces all of same dimension $\b_w$ (resp. $\mathrm{dim}\,\t^w+\mathrm{dim}\,\u_w$). \label{SS1}\end{proposition}
\bigskip

\subsection{Purity of cohomology}

\begin{proposition}The compactly supported cohomology of $\calZ$ and $\calY$ is pure and vanishes in odd degree.
\label{purity}\end{proposition}

\begin{proof}We prove it for $\calY$ but the proof is completely similar for $\calZ$. Let us choose $k_o$ a finite subfield of $k$ such that $G$, $B$ and $T$ are defined over $k_o$ and $T$ is split. Then $\mathrm{Gal}(k/k_o)$ acts trivially on $W$.
\bigskip

Recall (see \S\ref{SS}) that we have a partition into locally closed substacks
\begin{equation}
\calY=\coprod_{w\in W}\calY_w.
\end{equation}
By our choice of $k_o$ each stratum is defined over $k_o$.
\bigskip

We choose a total order $\{w_0,w_1,\dots\}$ on $W$ so that we have a decreasing filtration of closed substacks
$$
\calY=\calY_0\supset\calY_1\supset\cdots\supset\calY_{|W|-1}\supset\calZ_{|W|}=\emptyset
$$
satisfying $\calY_i\backslash\calY_{i+1}=\calY_{w_i}$ for all $i$. This defines a spectral sequence \cite[Chapter 6, Formula (2.5.2)]{SGA} from which we can reduce the proof of the purity of the cohomology of $\calY$ and the vanishing in odd degrees to the analogous statement for $\calY_w$.
\bigskip

The fibers of the projection $\calY_w\rightarrow\mathcal{O}_w$ are affine spaces all of same dimension $n(w)=\mathrm{dim}\,\t^w+\mathrm{dim}\,\u_w$ by Proposition \ref{SS} and so
$$
H_c^i(\calY_w,\Q)=H_c^{i-2n(w)}(\mathcal{O}_w,\Q)(-n(w)).
$$
Recall also that the map $\B(B_w)\rightarrow\mathcal{O}_w$ is an isomorphism and that
$$
H_c^{k+2\mathrm{dim}\, U_w}(\B(T),\Q)(\mathrm{dim}\, U_w)=H_c^k(\B(B_w),\Q).
$$
Therefore
$$
H_c^i(\calY_w,\Q)=H_c^{i-2\mathrm{dim}\,\t^w}(\B(T),\Q)(-\mathrm{dim}\,\t^w).
$$
Since the cohomology of $\B(T)$ is pure and vanishes in odd degree, the same is true for $H_c^i(\calY_w,\Q)$.
\end{proof}

\subsection{Restriction to the diagonal}

Notice that 

$$
H_c^i(\calZ,\Q)=H_c^i(\calG,p_!\Q\otimes p_!\Q)
$$
and so $H_c^i(\calZ,\Q)$ is naturally equipped with an action of  $W\times W$.
\bigskip

The aim of this section is to prove the following theorem.
\bigskip

\begin{theorem}The cohomological restriction

$$
H_c^{2i}(\calZ,\Q)\longrightarrow H_c^{2i}(\calB,\Q)
$$
of the diagonal morphism $\calB\rightarrow\calZ$ is $W$-equivariant for the Springer actions (where we consider the diagonal action of $W$ on the source).

\label{diag-res}

\end{theorem}

Let $q':\calB\rightarrow\t$ be induced by the canonical projection $\b\rightarrow\t$. Then

$$
H^i_c(\calZ,\Q)=H_c^i(\t\times_\car\t,(q',q')_!\Q),\hspace{1cm}H_c^i(\calB,\Q)=H_c^i(\t,q'_!\Q).
$$
Choose a total ordering $\{w_0,w_1,\dots\}$ on $W$ so that we have a decreasing filtration of closed substacks

$$
\calZ=\calZ_0\supset\calZ_1\supset\cdots\supset\calZ_{|W|-1}\supset\calZ_{|W|}=\emptyset
$$
satisfying $\calZ_i\backslash\calZ_{i+1}=\calZ_{w_i}$ for all $i$.
\bigskip

By \cite[Chapter 6, Formula (2.5.2)]{SGA} we have a spectral sequence

$$
E_1^{ij}=\calH^{i+j}\left(\left((q',q')|_{\calZ_{w_i}}\right)_!\Q\right)\Rightarrow \calH^{i+j}(q',q')_!\Q.
$$

For $w\in W$, consider the morphism 

$$
\Delta_{\t,w}:\t\longrightarrow \t\times_\car\t,\hspace{1cm}t\mapsto (w(t),t)
$$
and the following commutative diagram

$$
\xymatrix{\calZ_w\ar[r]\ar[d]_{q'_w}&\calZ\ar[d]^{(q',q')}\\
\t\ar[r]^-{\Delta_{\t,w}}&\t\times_\car\t}
$$
Then

$$
((q',q')|_{\calZ_w})_!\Q=\Delta_{{\t,w}\,*}q'_{w\,*}\Q=\Delta_{\t,w\,*}\Q\otimes R\Gamma_c(\B(T),\Q).
$$
As $H^{\rm odd}_c(\B(T),\Q)$ vanishes, the spectral sequence degenerates at $E_1$. 
\bigskip

\begin{proof}[Proof of Theorem \ref{diag-res}]Since the above spectral sequence degenerates at $E_1$, we have $\calH_c^{\rm odd}(q',q')_!\Q=0$ and $\calH^{\rm even}(q',q')_!\Q$ is a  successive extension of the $\Delta_{\t,w\,*}\Q\otimes H^{\rm even}_c(\B(T),\Q)$. By Proposition \ref{IE}, $\calH^{2i}(q',q')_!\Q$ is thus a perverse sheaf (up to a shift) that is the intermediate extension of its restriction to $\t_\reg\times_\car\t_\reg$.
\bigskip

We thus have

$$
\calH^{2i}(q',q')_!\Q=\Delta_*\Q\otimes H^{2i}_c(\B(T),\Q)
$$
where $\Delta:W\times\t\rightarrow\t\times_\car\t$, \,$(w,t)\mapsto (w(t),t)$ is the normalization morphism.
\bigskip

Using the spectral sequence

$$
H_c^i\left(\t\times_\car\t,\calH^j(q',q')_!\Q\right)\Rightarrow H_c^{i+j}(\calZ,\Q)
$$
together with the fact that $H_c^i(\t\times_\car\t,\Delta_*\Q)=0$ unless $i=2n$ we deduce that

$$
H_c^{-2i}(\calZ,\Q)=H_c^{2n}(\t\times_\car\t,\Delta_*\Q)\otimes H_c^{-2n-2i}(\B(T),\Q).
$$
The restriction morphism $H_c^i(\calZ,\Q)\rightarrow H_c^i(\calB,\Q)$ is induced by the restriction morphism

$$
\Delta_{\t,1}^*:H_c^{2n}(\t\times_\car\t,\Delta_*\Q)\rightarrow H_c^{2n}(\t,\Q)
$$
which is $W$-equivariant, hence the theorem.
\end{proof}

\section{Postnikov diagrams of  kernels}\label{kernel}

The aim of this section is to compute the kernel 

$$\calN:=(q,p)_!\Q$$
as well as its restriction to nilpotent elements.
\bigskip

To alleviate the notation we will denote by $f$ the representable morphism

$$
(q',p):\calB\rightarrow\calS
$$
considered in Theorem \ref{Steinpro}.
\bigskip

The map $(q,p)$ decomposes as
\begin{equation}
\xymatrix{\calB\ar[rr]^-{\delta:=(q,1_\calB)}&&\hat{\calB}:=\calT\times_\t\calB\ar[rr]^-{\hat{f}:=1_\calT\times p}&&\hat{\calS}:=\calT\times_\car\calG}
\label{DIAGvert}\end{equation}
or equivalently as
$$
\xymatrix{\calB\ar[rr]^-{\delta:=(\tau,1_\calB)}&&\B(T)\times\calB\ar[rr]^-{1_{\B(T)}\times f}&&\B(T)\times\calS}.
$$
under the identification $\calT=\t\times\B(T)$. Notice that $\tau$ corresponds to the $T$-torsor $\dot{\calB}=[\b/U]\rightarrow[\b/B]=\calB$.

The first map is a $T$-torsor which fits into the cartesian diagram
$$
\xymatrix{\calB\ar[r]^\tau\ar[d]_\delta\ar[r]&\B(T)\ar[r]\ar[d]&\mathrm{Spec}(k)\ar[d]\\
\hat{\calB}\ar[r]&\B(T)\times\B(T)\ar[r]& B(T)}
$$
where the map $\B(T)\rightarrow\B(T)\times\B(T)$ is the diagonal embbeding and $\B(T)\times\B(T)\rightarrow\B(T)$ is induced by $T\times T\rightarrow T, (t,h)\mapsto th^{-1}$.
\bigskip

\subsection{Postnikov diagrams associated to $T$-torsors}\label{Torsors}

Let $\calV$ be an   algebraic stack and $\dot{\calV}$ a $T$-torsor $\rho:\dot{\calV}\rightarrow\calV$. We want to compute $\rho_!\Q$.
\bigskip

This torsor corresponds to a morphism
$$
\sigma:\calV\rightarrow\B(T).
$$
We have a Chern class morphism
$$
c_\mathrm{can}:X^*(T)\rightarrow H^2(\B(T),\Q)(1)
$$
associated to the canonical $T$-torsor $\mathrm{Spec}(k)\rightarrow\B(T)$ which induces an isomorphism
$$
X^*(T)\otimes\Q\simeq H^2(\B(T),\Q)(1).
$$
Composed with $\sigma^*$, it induces a Chern class morphism
$$
c_\rho:X^*(T)\rightarrow H^2(\calV,\Q)(1).
$$
Concretely, if $\chi\in X^*(T)$, then $c_\rho(\chi)$ is the Chern class of the line bundle
$$
(\dot{\calV}\times\bbA^1)/T\rightarrow\calV,
$$
where $T$ acts on $\bbA^1$ as $x\cdot t=\chi(t^{-1})x$, for $(x,t)\in\bbA^1\times T$.
\bigskip

Since $H^2(\calV,\Q)=\ext^2_\calV(\Q,\Q)=\hom_\calV(\Q,\Q[2])$, we can regard $c$ as a map
$$
c_\rho:\Q\rightarrow X_*(T)\otimes\Q[2](1)
$$
in $\dcb(\calV)$.
\bigskip

This defines a complex of objects of  $\dcb(\calV)$
$$
\Q[-2n](-n)\overset{\partial_n}{\longrightarrow} X_*(T)\otimes\Q[2-2n](1-n)\overset{\partial_{n-1}}{\longrightarrow}\bigwedge^2X_*(T)\otimes\Q[4-2n](2-n)\overset{\partial_{n-2}}{\longrightarrow}\cdots
$$
\begin{equation}
\cdots\rightarrow\bigwedge^{n-1}X_*(T)\otimes\Q[-2](-1)\overset{\partial_1}{\longrightarrow}\bigwedge^nX_*(T)\otimes\Q,
\label{diagramhomo}
\end{equation}
where $\partial_{n-i+1}[2n-2i+2](n-i+1)$ is induced by the morphism
$$
\bigwedge^{i-1}X_*(T)\otimes\Q\longrightarrow\bigwedge^iX_*(T)\otimes\Q[2](1)
$$
that maps $v$ to $v\wedge c_\rho(1)$.
\bigskip

\begin{proposition}
(1) We have
$$
{^p}\calH^i(\rho_!\Q)=\calH^i(\rho_!\Q)\simeq\bigwedge^{2n-i}X_*(T)\otimes\Q(n-i).
$$
\noindent (2) The complex (\ref{diagramhomo}) is the base $\Lambda_b(\rho_!\Q)$ of the Postnikov diagram $\Lambda(\rho_!\Q)$.
\label{postnikovtore}
\end{proposition}

\begin{proof}
The proposition reduces to the case $n=1$ using an isomorphism $T\simeq (\mathbb{G}_m)^n$. Indeed, the complex (\ref{diagramhomo}) can be realized as the tensor product of analogous complexes associated to rank $1$ tori.
\end{proof}

\subsection{Postnikov diagram of  $\calN$}\label{Pos-K}

Since the map $\hat{f}:\hat{\calB}\rightarrow\hat{\calS}$ is a small resolution of singularities, we have
$$
\hat{f}_!\Q=\IC_{\hat{\calS}}.
$$
Following \S\ref{Torsors} we now regard the Chern class morphism
$$
c_\delta:X^*(T)\rightarrow H^2(\hat{\calB},\Q)(1)
$$
associated to the $T$-torsor $\delta:\calB\rightarrow\hat{\calB}$ as a morphism
$$
c_\delta:\Q\rightarrow X_*(T)\otimes\Q[2](1).
$$
Applying $\hat{f}_!$ we get
$$
\hat{f}_!(c_\delta):\IC_{\hat{\calS}}\rightarrow X_*(T)\otimes\IC_{\hat{\calS}}[2](1).
$$
Applying the functor $\hat{f}_!$ to the complex $\Lambda_b(\delta_!\Q)$ we obtain a complex
\begin{equation}
\IC_{\hat{\calS}}[-2n](-n)\overset{\partial_n}{\longrightarrow} X_*(T)\otimes\IC_{\hat{\calS}}[2-2n](1-n)\overset{\partial_{n-1}}{\longrightarrow}\cdots
\label{diagramhomoIC}
\end{equation}
\begin{flushright}
$\cdots\overset{\partial_1}{\longrightarrow}\left(\bigwedge^nX_*(T)\right)\otimes\IC_{\hat{\calS}}$,
\end{flushright}
where $\partial_{n-i+1}[2n-2i+2](n-i+1)$ is induced by the morphism
$$
\left(\bigwedge^{i-1}X_*(T)\right)\otimes\IC_{\hat{\calS}}\longrightarrow\left(\bigwedge^iX_*(T)\right)\otimes\IC_{\hat{\calS}}[2](1)
$$
given by $v\otimes s\mapsto v\wedge f_!(c_\delta)(s)$.
\bigskip

\bigskip

\begin{proposition}
(1) The perverse cohomology sheaves of $\calN$ are
$$
{^\p}\calH^i\left(\calN\right) =\left(\bigwedge^{2n-i}X_*(T)\right)\otimes\IC_{\hat{\calS}}(n-i).
$$

\noindent (2) We have $\Lambda(\calN)=\hat{f}_!\left(\Lambda(\delta_!\Q)\right)$. 

\noindent (3) Assume that $G$, $B$ and $T$ are defined over a finite subfield $k_o$ of $k$. The diagram $\Lambda(\calN)$ is the unique Postnikov diagram defined over $k_o$ whose base is the complex (\ref{diagramhomoIC}). 
\label{ladderunique}\end{proposition}

\begin{proof}
The assertion (1) follows from the fact that the morphism $\hat{f}$ is small and that the cohomology groups of $\Lambda(\delta_!\Q)$ are constant (see Proposition \ref{postnikovtore}(1)). The assertion (2) follows from Lemma \ref{Lambda}.
\bigskip

To prove assertion (3), we choose a finite subfield $k_o$ of $k$ on which $G$, $T$ and $B$ are defined and we let $T_o$ and $\hat{\calS}_o$ be the induced $k_o$-structure on $T$ and $\hat{\calS}$. We want to prove that $\hat{f}_{o!}\left(\Lambda(\delta_{o!}\Q)\right)$ is the unique Postnikov diagram that completes the complex (\ref{diagramhomoIC}) with $T$ and $\hat{\calS}$ replaced by $T_o$ and $\hat{\calS}_o$.

We need to prove that the hypothesis of Lemma \ref{laddercompletion} are satisfied. The condition (\ref{COND11}) is verified as
$$
\mathrm{Hom}(\IC_{\hat{\calS}_o}[-2j](-j),\IC_{\hat{\calS}_o}[-2i](-i))=\mathrm{Ext}^{2(j-i)}(\IC_{\hat{\calS}_o},\IC_{\hat{\calS}_o}(j-i))=0
$$
for all $j<i$.

The proof of the condition (\ref{COND21}) reduces to prove that
$$
\mathrm{Hom}(\IC_{\hat{\calS}_o}[-2j](-j),\IC_{\hat{\calS}_o}[-2i-k](-i))=\mathrm{Ext}^{2(j-i)-k}(\IC_{\hat{\calS}_o},\IC_{\hat{\calS}_o}(j-i))=0
$$
for all $j>i$ and $k\geq 1$. But this is clear from the next proposition using the exact sequence (\ref{BBD}).
\end{proof}

\begin{proposition}
$\mathrm{Ext}^j_{\hat{\calS}}(\IC_{\hat{\calS}},\IC_{\hat{\calS}})=0$ if $j$ is odd and $\mathrm{Ext}^{2i}_{\hat{\calS}}(\IC_{\hat{\calS}},\IC_{\hat{\calS}})$ is pure of weight $2i$.
\label{Ext}\end{proposition}

\begin{proof}
We need to compute
$$
\mathrm{Ext}^i_{\hat{\calS}}(\IC_{\hat{\calS}},\IC_{\hat{\calS}})=\mathrm{Hom}(\IC_{\hat{\calS}},\IC_{\hat{\calS}}[i]).
$$
Since $\hat{\calS}=\B(T)\times\calS$ and since the statement of the proposition is true if we replace $\hat{S}$ by $\B(T)$ (noticing that $\IC_{\B(T)}=\Q$), by K\"unneth formula we are reduced to prove the proposition with $\calS$ instead of $\hat{\calS}$.
\bigskip

Consider the cartesian diagram
$$
\xymatrix{\calY=\calB\times_\calS\calB\ar[r]^-{\pr_2}\ar[d]_{\pr_1}&\calB\ar[d]^f\\
\calB\ar[r]^f&\calS}
$$
Since $f$ is representable, the morphism $\pr_1$ is also representable. We have

\begin{align*}
\mathrm{Hom}(\IC_\calS,\IC_\calS[i])&=\mathrm{Hom}(f_!\Q,f_!\Q[i])\\
&=\mathrm{Hom}(f^*f_!\Q,\Q[i])\\
&=\mathrm{Hom}(\pr_{1!}\Q,\Q[i])\\
&=\mathrm{Hom}(\Q,\pr_1^!\Q[i])\\
&=H^i(\calY,\pr_1^!\Q)\\
&=H_c^{-i}(\calY,\Q),
\end{align*}
where the last identity is by Poincar\'e duality (notice that $\mathrm{dim}\,\calY=0$).
\bigskip

We thus conclude from Proposition \ref{purity}. 
\end{proof}

\subsection{Restriction to nilpotent elements}\label{nil-diag}

We let $\frak{z}$ be a locally closed subset of $z(\frak{g})$.  We consider the Lusztig correspondence 

$$
\xymatrix{&&\calB_\frak{z}:=[(\frak{z}+\frak{u})/B]\ar[rrd]^{p_\frak{z}}\ar[lld]_{q_\frak{z}}\ar[d]^{(q_\frak{z},p_\frak{z})}&&\\
\calT_\frak{z}:=[\frak{z}/T]&&\hat{\calS}_\frak{z}:=\B(T)\times[(\frak{z}+\g_{\rm nil})/G]\ar[rr]^-{\pr_{\calG_\frak{z}}}\ar[ll]_-{\pr_{\calT_\frak{z}}}&&\calG_\frak{z}:=[(\frak{z}+\g_{\rm nil})/G]}
$$
over $\frak{z}=\frak{z}/\!/W\subset \car$.
\bigskip

The $T$-torsor $\delta_\frak{z}:\calB_\frak{z}\rightarrow\hat{\calB}_\frak{z}=\B(T)\times\calB_\frak{z}$ induced from $\delta$ gives rise to a Chern class morphism

$$
c_{\delta_\frak{z}}:\Q\rightarrow X_*(T)\otimes\Q[2](1)
$$
and so following the same lines as in \S \ref{Pos-K}  we end up with a complex

\begin{equation}
\left(\overline{\mathbb{Q}}_{\ell,\,\B(T)}\boxtimes p_{\frak{z}\,!}\Q\right)[-2n](-n)\rightarrow X_*(T)\otimes \left(\overline{\mathbb{Q}}_{\ell,\,\B(T)}\boxtimes p_{\frak{z}\,!}\Q\right)[2-2n](1-n)\rightarrow\cdots
\label{complex-nil}\end{equation}
\begin{flushright}
$\cdots\longrightarrow\left(\bigwedge^nX_*(T)\right)\otimes \left(\overline{\mathbb{Q}}_{\ell,\,\B(T)}\boxtimes p_{\frak{z}\,!}\Q\right)$.
\end{flushright}
which is the restriction of the complex (\ref{diagramhomoIC}) to $\hat{\calS}_\frak{z}$.
\bigskip

If $\frak{z}=0$ we will replace the subscript $\frak{z}$ by nil.
\bigskip

\noindent Let $\pr_{\rm nil}:\hat{\calS}_\frak{z}\rightarrow\hat{\calS}_{\rm nil}=\B(T)\times\calG_{\rm nil}$ be the morphism induced by the projection $\frak{z}+\g_{\rm nil}\rightarrow\g_{\rm nil}$.

\begin{proposition} (1) We have

$$
(q_\frak{z},p_\frak{z})_!\Q=\pr_{\rm nil}^*\left((q_{\rm nil},p_{\rm nil})_!\Q\right).
$$
(2) Assume also that $G, B$ and $T$ are defined over  a finite subfield $k_o$ of $k$. The Postnikov diagram $\Lambda((q_{\rm nil},p_{\rm nil})_!\Q)=\hat{f}_{{\rm nil}\, !}\big(\Lambda(\delta_{{\rm nil}\, !}\Q)\big)$ is the unique Postnikov diagram defined over $k_o$ whose base is the complex (\ref{complex-nil}) with $\frak{z}=0$.
\end{proposition}

\begin{proof} The assertion (1) is clear. The proof of  assertion (2) goes exactly along the same lines as the proof of Proposition \ref{ladderunique}. 
\end{proof}

\section{Descent}\label{descent}

As explained in \S \ref{calcul}, the kernel $\calN$ does not descend to $\ocalT\times_\car\calG$. However we will prove in this section that we have a descent result if we restrict ourself to strata of $\car$ that we are now defining.

\subsection{The main theorem}

We have a natural stratification of $\car$ into smooth locally closed subsets indexed by the set $\frak{L}$ of $G$-conjugacy classes of Levi subgroups (of parabolic subgroups) of $G$

$$
\car=\coprod_{\lambda\in\frak{L}}\car_\lambda.
$$
Concretely, let $L$ be a representative of $\lambda\in\frak{L}$ containing $T$, then $\car_\lambda$ is the image of 

$$
z(\frak{l})^o:=\{x\in \frak{l}\,|\,C_G(x)=L\}\subset z(\frak{l})\subset \t
$$
by the quotient map $\t\rightarrow\car$. 
\bigskip

If $\calX$ is either $\calT,\calB,\hat{\calB}, \calS, \hat{\calS}, \ocalT $ or $\calG$ we put

$$
\calX_\lambda:=\calX\times_\car\car_\lambda
$$
for any $\lambda\in\frak{L}$.

\bigskip

We consider the Lusztig correspondence over $\lambda\in \frak{L}$

$$
\xymatrix{\calT_\lambda&&\calT_\lambda\times_{\car_\lambda}\calG_\lambda\ar[ll]_-{\pr_{\calT_\lambda}}\ar[rr]^-{\pr_{\calG_\lambda}}&&\calG_\lambda}.
$$

\bigskip

Notice that $\calT_\lambda=[\t_\lambda/T]\simeq \t_\lambda\times\B(T)$ with

$$
\t_\lambda=\coprod_{g\in N_G(T)/N_G(L,T)}g\,z(\frak{l})^og^{-1}.
$$
Notice that if for some $g\in G$ and $s\in z(\frak{l})^o$ we have ${\rm Ad}(g)(s)\in z(\frak{l})^o$, then $g\in N_G(L)$ \cite[Lemma 2.6.16]{Letellier}.
\bigskip

We thus  have isomorphisms of stacks

$$
\calG_\lambda\simeq [(z(\frak{l})^o+\frak{l}_{\rm nil})/N_G(L)],\hspace{1cm}\overline{\calT}_\lambda=[\t_\lambda/N]\simeq [z(\frak{l})^o/N_G(L,T)].
$$
Put $B^L:=B\cap L$ and let $\u^L$ be the Lie algebra of the unipotent radical of $B^L$. We consider the following commutative diagram (as in \S \ref{nil-diag} with $G=L$ and $\frak{z}=z(\frak{l})^o$)

\begin{equation}
\xymatrix{&[(z(\frak{l})^o+\u^L)/B^L]\ar[rd]^{p^L}\ar[ld]_{q^L}\ar[d]^{(q^L,p^L)}&\\
 [z(\frak{l})^o/T]\ar[d]& [z(\frak{l})^o/T]\times_{z(\frak{l})^o} [(z(\frak{l})^o+\frak{l}_{\rm nil})/L]\ar[r]\ar[l]\ar[d]&[(z(\frak{l})^o+\frak{l}_{\rm nil})/L]\ar[d]\\
 \overline{\calT}_\lambda&\overline{\calT}_\lambda\times_{\car_\lambda}\calG_\lambda\ar[r]\ar[l]&\calG_\lambda}
\label{diag-nil}\end{equation}
Put

$$
W^L:=N_L(T)/T,\hspace{1cm}\widehat{W}^L:=N_G(L,T)/T,\hspace{1cm}W_L:=N_G(L)/L
$$
We have an exact sequence

$$
1\rightarrow W^L\rightarrow \widehat{W}^L\rightarrow W_L\rightarrow 1
$$
and the group $\widehat{W}^L$ acts on

$$
 [z(\frak{l})^o/T]\times_{z(\frak{l})^o} [(z(\frak{l})^o+\frak{l}_{\rm nil})/L]
$$
in the natural way on the first factor and via $W_L$ on the second factor.
\bigskip

The first two vertical arrows of the diagram (\ref{diag-nil}) are $\widehat{W}^L$-torsors while the right vertical arrows is a $W_L$-torsor.
\bigskip

\begin{theorem}The complex $\calN^L:=(q^L,p^L)_!\Q$ is $\widehat{W}^L$-equivariant, i.e. it descends to a complex $\bN_\lambda$ on  $\overline{\calT}_\lambda\times_{\car_\lambda}\calG_\lambda$.
\label{maintheo-descent}\end{theorem}

\bigskip

\begin{remark} We have the following commutative diagram

$$
\xymatrix{\calB_\lambda\ar[d]_{(q_\lambda,p_\lambda)}&&[(z(\frak{l})^o+\u^L)/B^L]\ar[d]^{(q^L,p^L)}\ar[ll]\\
\calT_\lambda\times_{\car_\lambda}\calG_\lambda\ar[d]&&[z(\frak{l})^o/T]\times_{z(\frak{l})^o} [(z(\frak{l})^o+\frak{l}_{\rm nil})/L]\ar[ll]\ar[dll]\\
\ocalT_\lambda\times_{\car_\lambda}\calG_\lambda&&}
$$
where the top square is cartesian. From this diagram we see that

$$
\bN_\lambda|_{\calT_\lambda\times_\car\calG_\lambda}\simeq\calN|_{\calT_\lambda\times_\car\calG_\lambda}=(q_\lambda,p_\lambda)_!\Q
$$
where $\calN$ is as in \S \ref{kernel}. In other words, the complex $\calN_\lambda:=\calN|_{\calT_\lambda\times_\car\calG_\lambda}$ is $W$-equivariant.

\label{rem33}\end{remark}

\subsection{Proof of Theorem \ref{maintheo-descent}}\label{main-descent}

The essential case is when $L=G$ which case reduces to the nilpotent elements (see \S \ref{nil-diag}). We thus consider the following diagram 

$$
\xymatrix{&&\calB_{\rm nil}=[\frak{u}/B]\ar[rrd]^{p_{\rm nil}}\ar[lld]_{q_{\rm nil}}\ar[d]^{(q_{\rm nil},p_{\rm nil})}&&\\
\calT_{\rm nil}=\B(T)&&\hat{\calS}_{\rm nil}=\B(T)\times[\g_{\rm nil}/G]\ar[rr]^-{\pr_{\calG_{\rm nil}}}\ar[ll]_-{\pr_{\calT_{\rm nil}}}&&\calG_{\rm nil}=[\g_{\rm nil}/G]}
$$
and we want to prove that the complex $(q_{\rm nil},p_{\rm nil})_!\Q$ is $W$-equivariant and so descends to $\B(N)\times\calG_{\rm nil}$.
\bigskip

Consider now the diagram

$$
\xymatrix{&&\calB\ar[rrd]^p\ar[lld]_{\tau}\ar[d]^{(\tau,p)}&&\\
\B(T)&&\B(T)\times\calG\ar[rr]^-{\pr_{\calG}}\ar[ll]_-{\pr_{\B(T)}}&&\calG}
$$
Notice that the complex $(q_{\rm nil},p_{\rm nil})_!\Q$ is the restriction of $(\tau,p)_!\Q$ to $\B(T)\times\calG_{\rm nil}$. 
\bigskip

It is thus enough to prove that $(\tau,p)_!\Q$ is $W$-equivariant.

\bigskip

The map $(\tau,p)$ decomposes as 

$$
\xymatrix{\calB\ar[rr]^-{\delta}&&\hat{\calB}=\B(T)\times\calB\ar[rr]^-{\hat{p}:=1_{\B(T)}\times p}&&\hat{\calG}:=\B(T)\times\calG}
$$

Following the strategy of \S \ref{Torsors}, we see that the complex $(\tau,p)_!\Q$ is the outcome of the Postnikov diagram $\hat{p}_!\Lambda(\delta_!\Q)$.  The base of this Postnikov diagram is the complex

\begin{equation}
\IC_{\hat{\calG}}(\hat{\mathcal{L}})[-2n](-n)\overset{\partial_n}{\longrightarrow} X_*(T)\otimes\IC_{\hat{\calG}}(\hat{\mathcal{L}})[2-2n](1-n)\overset{\partial_{n-1}}{\longrightarrow}\cdots
\label{diagramIC}
\end{equation}
\begin{flushright}
$\cdots\overset{\partial_1}{\longrightarrow}\left(\bigwedge^nX_*(T)\right)\otimes\IC_{\hat{\calG}}(\hat{\mathcal{L}})$,
\end{flushright}
where $\hat{\mathcal{L}}$ is the semisimple local system $\Q\boxtimes \mathcal{L}$ with $\mathcal{L}=p_{\rss\, !}\Q$ where $p_\rss:\calB_\rss\rightarrow\calG_\rss$ is the restriction of $p$ to semisimple regular elements.
\bigskip

We consider on  vertices  of (\ref{diagramIC}) the $W$-action given by the Springer action  on $\IC_\calG(\mathcal{L})=p_!\Q$ and the obvious one on $X_*(T)$.

\begin{theorem}
The arrows of the complex (\ref{diagramIC}) are $W$-equivariant.
\label{Step1}\end{theorem}
\bigskip

Theorem \ref{Step1} will be a consequence of Theorem \ref{maintheo2} below.
\bigskip

\begin{corollary}
The complex (\ref{diagramIC}) descends to a unique complex $\overline{\Lambda}_b$ in $\dcb(\hat{\calG})$ whose vertices are
$$
\left(\bigwedge^{2n-i}X_*(T)\right)\otimes\left(\overline{\mathbb{Q}}_{\ell,\, \B(N)}\boxtimes \IC_\calG(\mathcal{L})\right)[2n-2i](n-i).
$$
If $k_o$ is a finite subfield of $k$ on which $G, T$ and $B$ are defined, then $\overline{\Lambda}_b$ is naturally defined over $k_o$.
\label{coromain}\end{corollary}

\begin{theorem}
The complex $\overline{\Lambda}_b$ of Corollary \ref{coromain} can be completed into a unique Postnikov diagram $\overline{\Lambda}$ (up to a unique isomorphism) such that for any finite subfield $k_o$ of $k$ on which $G,T$ and $B$ are defined, the diagram $\overline{\Lambda}$ is also defined over $k_o$ with $\overline{\Lambda}_b$ equipped with its natural $k_o$-structure.
\bigskip

We have
$$
\overline{\Lambda}|_{\hat{\calG}}\simeq\hat{p}_!\Lambda(\delta_!\Q).
$$
In particular the outcome $(\tau,p)_!\Q$ of $\hat{p}_!\Lambda(\delta_!\Q)$ descends naturally to $\B(N)\times\calG$.
\label{Maintheo1}\end{theorem}

\begin{proof}
We choose a finite subfield $k_o$ on which $G$, $T$ and $B$ are defined. By the proof of Proposition \ref{ladderunique} , the conditions of Proposition \ref{W-prop} are satisfied over $k_o$. Therefore by Proposition \ref{W-prop} there is a unique Postnikov diagram $\overline{\Lambda}_o$ that completes the natural descent of $\overline{\Lambda}_b$ to $k_o$. The wanted Postnikov diagram $\overline{\Lambda}$ is the one induced by $\overline{\Lambda}_o$ from $k_o$ to $k$.
\end{proof}
\bigskip

As explained in \S \ref{W-equivariant}, the group $W$ acts on 

$$
\ext^i_{\hat{\calG}}\left(\IC_{\hat{\calG}}(\hat{\mathcal{L}}),\IC_{\hat{\calG}}(\hat{\mathcal{L}})\right)={\rm Hom}\left(\IC_{\hat{\calG}}(\hat{\mathcal{L}}),\IC_{\hat{\calG}}(\hat{\mathcal{L}})[i]\right).
$$

\bigskip

\begin{theorem}
The map
$$
X^*(T)\overset{c_{\delta}}{\longrightarrow} H^2(\hat{\calB},\Q)(1)=\ext^2_{\hat{\calB}}(\Q,\Q)(1)\longrightarrow\ext^2_{\hat{\calG}}\left(\IC_{\hat{\calG}}(\hat{\mathcal{L}}),\IC_{\hat{\calG}}(\hat{\mathcal{L}})\right)(1),
$$
where the first arrow is the Chern class morphism of $\delta$ and the second arrow is given by $\hat{p}_!$, is $W$-equivariant.
\label{maintheo2}\end{theorem}
\bigskip

This is \emph{apriori} not obvious as the natural action of $W$ on $\hat{\calG}$ does not lift to $\hat{\calB}$.
\bigskip

We have the following proposition.

\begin{proposition}
The Chern class morphism
$$
c_\delta:X^*(T)\rightarrow H^2(\hat{\calB},\Q)(1)
$$
of the $T$-torsor $\delta$ is $W$-equivariant.
\label{WChern}\end{proposition}

\begin{proof}
As $\hat{\calB}=\B(T)\times\calB$, by K\"unneth formula we have
$$
H^2(\hat{\calB},\Q)(1)= H^2(\calB,\Q)(1)\oplus H^2(\B(T),\Q)(1).
$$
Then
$$
c_\delta=c_\pi\oplus -c_\mathrm{can},
$$
where $\pi$ is the $T$-torsor $\dot{\calB}\rightarrow \calB$. The Chern class morphism $c_\mathrm{can}$ of the trivial $T$-torsor $\mathrm{Spec}(k)\rightarrow\B(T)$ is $W$-equivariant.

We thus need to prove that $c_\pi$ is $W$-equivariant. 
\bigskip

The Chern class morphism $c_\pi$  decomposes as
$$
X^*(T)\rightarrow H^2(\B(T),\Q)\simeq H^2(\B(B),\Q)\rightarrow H^2(\calB,\Q),
$$
where the first map is $c_\mathrm{can}$ and the second one is the cohomological restriction  of the canonical map $\calB\rightarrow\B(B)$. It is $W$-equivariant by Lemma \ref{canW}.
\end{proof}

Theorem \ref{maintheo2} is a consequence of Proposition \ref{WChern} together with the following one.

\begin{proposition}
The map
\begin{equation}
H^2(\hat{\calB},\Q)(1)=\ext^2_{\hat{\calB}}(\Q,\Q)(1)\rightarrow\ext^2_{\hat{\calG}}\left(\IC_{\hat{\calG}}(\hat{\mathcal{L}}),\IC_{\hat{\calG}}(\hat{\mathcal{L}})\right)(1)
\label{Func}\end{equation}
induced by $\hat{p}_!$ is $W$-equivariant.
\label{Z}\end{proposition}

\begin{proof} By K\"unneth formula we are reduced to prove that the same statement is true with $\hat{\calB}$ and $\hat{\calG}$ replaced by $\calB$ and $\calG$, i.e. we need to see that the map

\begin{equation}
H^2(\calB,\Q)(1)=\ext^2_{\calB}(\Q,\Q)(1)\rightarrow\ext^2_{\calG}\left(\IC_\calG(\mathcal{L}),\IC_\calG(\mathcal{L})\right)(1)
\label{Func1}\end{equation}
given by the functor $p_!$ is $W$-equivariant.
\bigskip

Consider the cartesian diagram

$$
\xymatrix{\calZ\ar[r]^-{\pr_2}\ar[d]_-{\pr_1}&\calB\ar[d]^p\\
\calB\ar[r]^p&\calG}
$$
We have
\begin{align*}
\ext^i_\calG(\IC_\calG(\mathcal{L}),\IC_\calG(\mathcal{L}))&={\rm Hom}(p_!\Q,p_!\Q[i])\\
&={\rm Hom}(p^*p_!\Q,\Q[i])\\
&=H^i(\calZ,\pr_1^!\Q)\\
&=H^{-i}_c(\calZ,\Q).
\end{align*}

The dual of the map (\ref{Func1}) is the cohomological restriction 

$$
H_c^{-2}(\calZ,\Q)\longrightarrow H_c^{-2}(\calB,\Q)
$$
of the diagonal morphism $\calB\rightarrow\calZ$. The proposition is thus a consequence of Theorem \ref{diag-res}.
\end{proof}

\subsection{Descent over regular elements}\label{descentrege}

We consider the diagram (\ref{DIAGvert}) over regular elements
\bigskip
$$
\xymatrix{\calB_\reg\ar[rr]^{\delta_\reg}\ar@/^2.0pc/@[red][rrrr]^{(q,p)_\reg}&&\hat{\calB}_\reg\ar[rr]^{\hat{f}_\reg}&&\hat{\calS}_\reg}
$$
and put 
$$
\calN_\reg:=(q,p)_{\reg\, !}\Q=\calN|_{\hat{\calS}_\reg}.
$$
\begin{theorem}The kernel $\calN_\reg$ is $W$-equivariant, i.e. it descends to $\ocalT\times_\car\calG_\reg$.
\end{theorem}

\begin{proof}Following the strategy of \S \ref{main-descent}, we need to prove that the base $\Lambda_b(\calN_\reg)$  of the Postnikov diagram $\Lambda(\calN_\reg)$ is $W$-equivariant. The map $\hat{f}_\reg$ being an isomorphism, the vertices of $\Lambda_b(\calN_\reg)$ are constant sheaves (up to a shift) and so are clearly $W$-equivariant. To prove that the arrows are $W$-equivariant, we need to prove (analogously to Theorem \ref{maintheo2}) that the map

$$
X^*(T)\longrightarrow H^2(\hat{\calB}_\reg,\Q)(1)\longrightarrow\ext^2(\hat{f}_{\reg\, !}\Q,\hat{f}_{\reg\, !}\Q)(1)
$$
is $W$-equivariant. The first map being $W$-equivariant (see Proposition \ref{WChern}), we need to see that  the second map too. However, notice that $\hat{f}_\reg$ is an isomorphism and so the second map is an isomorphism

$$
H^2(\hat{\calB}_\reg,\Q)\longrightarrow H^2(\hat{\calS}_\reg,\Q)
$$
which is dual to the restriction morphism

$$
H^{-2}_c(\hat{\calS}_\reg,\Q)\longrightarrow H^{-2}_c(\hat{\calB}_\reg,\Q).
$$
By K\"unneth formula we are reduced to prove that the restriction map (which is an isomorphism)

$$
H_c^{-2}(\calS_\reg,\Q)\longrightarrow H^{-2}_c(\calB_\reg,\Q)
$$
is $W$-equivariant. Notice that $H_c^{-2}(\calS_\reg,\Q)$ and $H^{-2}_c(\calB_\reg,\Q)$ are both the hypercohomology of the proper pushforwards of $\Q$ along the maps $\pr_2:\calS_\reg\rightarrow\calG_\reg$ and $p_\reg:\calB_\reg\rightarrow\calG_\reg$ respectively and these pushforwards are the intermediate extensions the pusforwards over semisimple regular elements. The result follows from the fact that $W$ acts on $\calB_\rss$ and that the isomorphism $\calB_\rss\simeq \calS_\rss$ is $W$-equivariant.

\end{proof}

\subsection{"Un calcul triste"}\label{calcul}

In this section we give an example where the kernel $\calN$ is not $W$-equivariant.

\begin{nothing}Notice that if $\calN$ were $W$-equivariant, then by Remark \ref{W-Hom}, the diagram $\Lambda(\calN)$ would be also $W$-equivariant and by Proposition \ref{ladderunique}(2) so would be the arrows of the complex $\Lambda_b(\calN)$. \end{nothing}
\bigskip

For $w\in W$ put $\calB_w:=[{\rm Ad}(w)\b/wBw^{-1}]$. We then have a cartesian diagram

$$
\xymatrix{\calB\ar[d]_{(q',p)}\ar[rr]^w&&\calB_w\ar[d]^{(q'_w,p_w)}\\
\calS\ar[rr]^w&&\calS}
$$
from which we get a commutative diagram (base change)

$$
\xymatrix{H^2(\calB_w,\Q)\ar[d]_{w^*}\ar[rr]^-{a_w}&&\ext^2\left((q'_w,p_w)_!\Q,(q'_w,p_w)_!\Q\right)\ar[d]^{w^*}\\
H^2(\calB,\Q)\ar[rr]^-a&&\ext^2\left((q',p)_!\Q,(q',p)_!\Q\right)}
$$
where the horizontal arrows are given by the functors $(q'_w,p_w)_!$ and $(q,p)_!$ respectively. Identifying $H^2(\calB,\Q)$ and $H^2(\calB_w,\Q)$ with $IH^2(\calS,\Q)$ on one hand, and $(q',p)_!\Q$ and $(q'_w,p_w)_!\Q$ with $\IC_\calS$ on the other hand, we end up with a commutative diagram

$$
\xymatrix{IH^2(\calS,\Q)\ar[d]_-{w^*}\ar[rr]^-{a_w}&&\ext^2(\IC_\calS,\IC_\calS)\ar[d]^{w^*}\\
IH^2(\calS,\Q)\ar[rr]^-a&&\ext^2(\IC_\calS,\IC_\calS)}
$$
Similarly to \S \ref{main-descent}, we see that that the arrows of the diagram $\Lambda_b(\calN)$ are $W$-equivariant if and only if $a=a_w$ for all $w\in W$.

\bigskip

\begin{nothing} Assume that $G=\SL_2$ and put $a_-:=a_w$ for $w\neq 1$ and $a_+:=a$. We prove below that $a_+\neq a_-$  proving thus the non-equivariance of $\calN$.
\bigskip
\end{nothing}

Note that

\begin{align*}
&S:=\t\times_\car\g=\{(x,t)\in \sl_2\times\bbA^1\,|\, \det(x)=-t^2\}\\
&X_+=\{(x,t,D)\in\sl_2\times\bbA^1\times\mathbb{P}^1\,|\, (x,t)\in S, D\subset{\rm Ker}(x-t{\rm Id})\}
\end{align*}
and denote by $\rho_+:X_+\rightarrow S$ the map $(x,t,D)\mapsto (x,t)$. We consider also the analogous map $\rho_-:X_-\rightarrow S$ where we replace ${\rm Ker}(x-t{\rm Id})$ in the definition of $X_+$ by ${\rm Ker}(x+t{\rm Id})$. 

Note that if $B$ denotes the Borel subgroup of upper triangular matrices, then 

$$
X_+\simeq \{(x,t,gB)\in\sl_2\times\bbA^1\times \SL_2/B\,|\, g^{-1}xg\in(t,-t)+\u\}
$$
and $X_-$ corresponds to choosing the opposite Borel subgroup of lower triangular matrices.

We have compactifications

\begin{align*}
&\overline{S}:=\{(X;T;Z)\in\mathbb{P}^4\,|\, \det(X)=-T^2\}\\
&\overline{X}_+:=\{((X;T;Z),D)\in\overline{S}\times\mathbb{P}^1\,|\, D\subset{\rm Ker}(X-T{\rm Id})\}\\
&\overline{X}_-:=\{((X;T;Z),D)\in\overline{S}\times\mathbb{P}^1\,|\, D\subset{\rm Ker}(X+T{\rm Id})\}\
\end{align*}
where we identify $\sl_2$ with $\bbA^3$ in the obvious way, and the small resolutions $\overline{\rho}_+:\overline{X}_+\rightarrow\overline{S}$ and $\overline{\rho}_-:\overline{X}_-\rightarrow\overline{S}$.

Identifying $IH^2(\overline{S},\Q)$ with $H^2(\overline{X}_+,\Q)$ (resp. with $H^2(\overline{X}_-,\Q)$) and $\IC_{\overline{S}}$ with $\overline{\rho}_{+*}\Q$ (resp. $\overline{\rho}_{-\, *}\Q$) we get a map $$\overline{a}_+:IH^2(\overline{S},\Q)\rightarrow\ext^2\left(\IC_{\overline{S}},\IC_{\overline{S}}\right)$$(resp. we get a map  $\overline{a}_-:IH^2(\overline{S},\Q)\rightarrow\ext^2\left(\IC_{\overline{S}},\IC_{\overline{S}}\right)$).
\bigskip

If we compose $\overline{a}_+$ (resp. $\overline{a}_-$) with the natural map

$$
\ext^2\left(\IC_{\overline{S}},\IC_{\overline{S}}\right)\rightarrow{\rm Hom}\left(IH^2(\overline{S},\Q),IH^4(\overline{S},\Q)\right)
$$
we get the cup-product 

$$
IH^2(\overline{S},\Q)\otimes IH^2(\overline{S},\Q)\overset{\smile_+}{\longrightarrow} IH^4(\overline{S},\Q)
$$
induced by the one on $H^*(\overline{X}_+,\Q)$ (resp. we get the cup-product $\smile_-$ induced by the one on $H^*(\overline{X}_-,\Q)$).

We have the following result \cite{Verdier}.

\begin{proposition}[Verdier] The two cup-products 

$$
\xymatrix{IH^2(\overline{S},\Q)\otimes IH^2(\overline{S},\Q)\ar@<-3pt>[r]_-{\smile_-}\ar@<3pt>[r]^-{\smile_+}& IH^4(\overline{S},\Q)}
$$
are not the same.

\end{proposition}

\begin{proof} For the convenience of the reader we outline the proof. 

Consider the following commutative diagram

$$
\xymatrix{&&\mathbb{P}^1\times\mathbb{P}^1\ar[dll]_{\pr_1}\ar[drr]^{\pr_2}&&\\
\mathbb{P}^1&&\overline{X}_+\times_{\overline{S}}\overline{X}_-\ar[drr]^{\pr_-}\ar[dll]_{\pr_+}\ar[u]^f&&\mathbb{P}^1\\
\overline{X}_+\ar[rrd]^{\overline{\rho}_+}\ar[u]^{f_+}&&&&\overline{X}_-\ar[lld]_{\overline{\rho}_-}\ar[u]_{f_-}\\
&&\overline{S}&&}
$$
Consider the divisors

\begin{align*}
&D_{1,+}:=\pr_+\left(f^{-1}(\{\infty\}\times\mathbb{P}^1)\right)=f_+^{-1}(\infty),\hspace{.5cm}D_{2,+}:=\pr_+\left(f^{-1}(\mathbb{P}^1\times\{\infty\})\right)\\
&D_{1,-}:=\pr_-\left(f^{-1}(\mathbb{P}^1\times\{\infty\})\right)=f_-^{-1}(\infty),\hspace{.5cm}D_{2,-}:=\pr_-\left(f^{-1}(\{\infty\}\times\mathbb{P}^1)\right)
\end{align*}
Under the natural identifications

$$
H^2(\overline{X}_+,\Q)\simeq IH^2(\overline{S},\Q)\simeq H^2(\overline{X}_-,\Q)
$$
The class ${\rm cl}(D_{1,+/-})\in H^2(\overline{X}_{+/-},\Q)$ corresponds to ${\rm cl}(D_{2,-/+})\in H^2(\overline{X}_{-/+},\Q)$ and so the proposition follows from the calculations

$$
{\rm cl}(D_{1,+/-})\smile_{+/-}{\rm cl}(D_{1,+/-})=0,\hspace{.5cm}{\rm cl}(D_{2,+/-})\smile_{+/-}{\rm cl}(D_{2,+/-})\neq 0.
$$

\end{proof}

As a consequence of the proposition we get that the two maps $\overline{a}_+$ and $\overline{a}_-$ are different. 
\bigskip

Let us now deduce from this that $a_+$ and $a_-$ are also different.

Put 

$$
K:=\underline{\rm Hom}(\IC_{\overline{S}},\IC_{\overline{S}}).
$$
Notice that 

$$
K|_S=\underline{\rm Hom}(\IC_S,\IC_S).
$$
Put

\begin{align*}
&\overline{\sigma}:=\overline{a}_+-\overline{a}_-:IH^2(\overline{S},\Q)\rightarrow \ext^2(\IC_{\overline{S}},\IC_{\overline{S}})=H^2(\overline{S},K)\\
&\sigma:=a_+-a_-:IH^2(S,\Q)\rightarrow \ext^2(\IC_S,\IC_S)=H^2(S,K|_S).
\end{align*}
We have the following commutative diagram

$$
\xymatrix{IH^2(\overline{S},\Q)\ar[r]^{\overline{\sigma}}\ar[d]&H^2(\overline{S},K)\ar[d]^{\varphi_S}\\
IH^2(S,\Q)\ar[r]^{\sigma}&H^2(S,K|_S)}
$$
where the vertical arrows are the restriction maps.

Let $\overline{\alpha}\in IH^2(\overline{S},\Q)$ be such that $\overline{\sigma}(\overline{\alpha})\neq 0$ and denote by $\alpha\in IH^2(S,\Q)$ the image of $\overline{\alpha}$.

\begin{proposition}We have

$$
\sigma(\alpha)\neq 0
$$
and so $a_+\neq a_-$.

\end{proposition}

\begin{proof}We need to prove that 

\begin{equation}
\varphi_S\left(\overline{\sigma}(\overline{\alpha})\right)\neq 0.
\label{alpha}\end{equation}
We have the open covering
$$
\overline{S}=\overline{S}_\reg\cup S
$$
where $\overline{S}_\reg$ is the smooth open subset of elements $(X;T;Z)$ in $\overline{S}$ such that $X\neq 0$.
\bigskip

Notice that $S_\reg=\overline{S}_\reg\cap S$ and so we have the exact sequence (Mayer-Vietoris)

$$
\cdots\rightarrow H^1(S_\reg,K|_{S_\reg})\rightarrow H^2(\overline{S},K)\overset{\varphi}{\longrightarrow} H^2(\overline{S}_\reg,K|_{\overline{S}_\reg})\oplus H^2(S,K|_S)\rightarrow\cdots
$$
Since $S_\reg$ and $\overline{S}_\reg$ are smooth open subsets of $S$ and $\overline{S}$ respectively, we have

$$
K|_{S_\reg}=\underline{\rm Hom}_{S_\reg}(\Q,\Q)=\Q,\hspace{.5cm}\text{ and }\hspace{.5cm} K|_{\overline{S}_\reg}=\Q.
$$
Notice also that
$$
H^1(S_\reg,\Q)=0
$$
as $\rho|_{X_\reg}:X_\reg\simeq S_\reg$ and $H^1(X_\reg,\Q)\simeq H^1(X,\Q)=0$.
\bigskip

The map $\varphi$ is thus injective and so to prove (\ref{alpha}) we are reduced to prove that

$$
\varphi_{\overline{S}_\reg}(\overline{\sigma}(\overline{\alpha}))= 0
$$
where $\varphi_{\overline{S}_\reg}:H^2(\overline{S},K)\rightarrow H^2(\overline{S}_\reg,\Q)$ is the restriction map. 

As the restrictions of $\overline{\rho}_+$ and $\overline{\rho}_-$ over regular elements induce isomorphisms $\overline{X}_{+,\reg}\simeq \overline{S}_\reg$ and $\overline{X}_{-,\reg}\simeq\overline{S}_\reg$ we have the following commutative diagram
$$
\xymatrix{H^2(\overline{X}_+,\Q)\simeq IH^2(\overline{S},\Q)\ar[r]^-{\overline{a}_+}\ar[d]&H^2(\overline{S},K)\ar[d]^{\varphi_{\overline{S}_\reg}}\\
H^2(\overline{X}_{+,\reg},\Q)\simeq H^2(\overline{S}_\reg,\Q)\ar[r]^-{\rm Id}&H^2(\overline{S}_\reg,\Q)}
$$
and the analogous diagram with $\overline{X}_-$ instead of $\overline{X}_+$. From these two commutative diagrams we deduce that

$$
\varphi_{\overline{S}_\reg}\circ \overline{\sigma}=0.
$$

\end{proof}

\section{The functors $\I_\lambda$ and $\R_\lambda$}\label{I-R}

\subsection{Definition}

We consider the cohomological correspondence  where $\pr_i$ is the $i$-th projection
$$
\xymatrix{\ocalT_\lambda&&\ocalT_\lambda\times_{\car_\lambda}\calG_\lambda\ar[rr]^-{\pr_2}\ar[ll]_-{\pr_1}&&\calG_\lambda}.
$$
with kernel $\bN_\lambda$. 

We define the associated  restriction and induction functors
\begin{align*}
&\R_\lambda:\dcb(\calG_\lambda)\rightarrow\dcb(\ocalT_\lambda),\hspace{.5cm}K\mapsto\pr_{1!}\left(\pr_2^*(K)\otimes\bN_\lambda\right),\\
&\I_\lambda:\dcb(\ocalT_\lambda)\rightarrow\dcb(\calG_\lambda),\hspace{.5cm}K\mapsto \pr_{2*}\,\underline{\rm Hom}\left(\bN_\lambda,\pr_1^!(K)\right)
\end{align*}
as  in \S \ref{cor}. 

If we denote $\Ind_\lambda$ and $\Res_\lambda$ the induction and restriction functors associated to the correspondence 

$$
\xymatrix{\calT_\lambda&&\calT_\lambda\times_{\car_\lambda}\calG_\lambda\ar[rr]^-{\pr_2}\ar[ll]_-{\pr_1}&&\calG_\lambda}
$$
with kernel $\calN_\lambda=\calN|_{\calT_\lambda\times_\car\calG_\lambda}$, then by Remark \ref{rem33}

$$
\Ind_\lambda=\I_\lambda\circ\pi_{\lambda\, *},\hspace{1cm}\Res_\lambda=\pi_\lambda^*\circ\R_\lambda.
$$
Since  $q_\lambda^!=q_\lambda^*$ and $p_{\lambda\,!}=p_{\lambda\, *}$, we see that $\Ind_\lambda$ decomposes also as the composition of $\pi_{\lambda\, !}$ followed by the functor 

$$
\dcb(\ocalT_\lambda)\rightarrow\dcb(\calG_\lambda),\hspace{.5cm}K\mapsto \pr_{2!}\left(\pr_1^*(K)\otimes\bN_\lambda\right)
$$
and so by Remark \ref{Inv} we have

\begin{equation}
\I_\lambda(K)=\pr_{2!}\left(\pr_1^*(K)\otimes\bN_\lambda\right).
\label{form2}\end{equation}

\bigskip

As in \S \ref{perInd}, we define a pair of adjoint functors $({^p}\R_\lambda,{^p}\I_\lambda)$

$$
\xymatrix{\calM(\overline{\t}_\lambda)\ar@/_/[rr]_-{{^p}\I_\lambda}&&\calM(\calG_\lambda)\ar@/_/[ll]_-{{^p}\R_\lambda}}
$$
where $\overline{\t}_\lambda:=[\t_\lambda/W]$. We let $\overline{s}_\lambda:\ocalT_\lambda\rightarrow\overline{\t}_\lambda$ be the quotient map of  the identity map $\t_\lambda\rightarrow\t_\lambda$ by $N\rightarrow W$. 

\bigskip

\begin{lemma} We have

$$
{^p}\I_\lambda=\I_\lambda\circ \overline{s}_\lambda^![n](n),\hspace{.5cm}{^p}\R_\lambda={^p}\calH^0\circ \overline{s}_{\lambda\, !}\circ \R_\lambda [-n](-n).
$$

\label{compatible}\end{lemma}

\begin{proof} Since the functors $\overline{s}_\lambda^![n](n)$ and ${^p}\calH^0\circ \overline{s}_{\lambda\, !}[-n](-n)$ between $\calM(\overline{\t}_\lambda)$ and $\calM(\ocalT_\lambda)$ are inverse to each other, by adjunction we deduce the second equality from the first one. 

Analogously to the definition ${^p}\underline{\Ind}$ (see \S \ref{perInd}) we define ${^p}\underline{\Ind}_\lambda$ by

$$
{^p}\underline{\Ind}_\lambda:=\Ind_\lambda\circ s_\lambda^![n](n)={^p}\I_\lambda\circ \pi_{\t_\lambda\, !}.
$$
where 

$$
\xymatrix{\calT_\lambda\ar[rr]^{s_\lambda}\ar[d]_{\pi_\lambda}&&\t_\lambda\ar[d]^{\pi_{\t_\lambda}}\\
\ocalT_\lambda\ar[rr]^{\overline{s}_\lambda}&&\overline{\t}_\lambda}
$$

By Remark \ref{Inv}, to prove the first equality, we need  to see that

$$
{^p}\underline{\Ind}_\lambda=\I_\lambda\circ \overline{\pr}_{\t_\lambda}^![n](n)\circ  \pi_{\t_\lambda\, !}
$$
Using that $\Ind_\lambda=\I_\lambda\circ\pi_{\lambda\, !}$, we are reduced to prove that

$$
\overline{s}_\lambda^![n](n)\circ  \pi_{\t_\lambda\, !}=\pi_{\lambda\, !}\circ s_\lambda^![n](n).
$$
But this follows from the base change theorem as $\overline{s}_\lambda$ and $s_\lambda$ are smooth of same relative dimension.

\end{proof}

\subsection{Semisimple regular elements}\label{semisimplereg}

Let us see that $\I_\lambda$ and $\R_\lambda$ are the identity functors when $\lambda$ represents semisimple regular elements (i.e. when $\lambda$ is the conjugacy class of maximal tori) in which case we use the subscript $\rss$ instead of $\lambda$ in the notation.

Recall that
$$
U\times\t_{\reg}\longrightarrow\b_{\rss},\hspace{.5cm}(u,t)\mapsto \mathrm{Ad}(u)(t)
$$
is an isomorphism and so taking the quotient by $B$ on both sides we end up with an isomorphism
$$
\calT_{\reg}\overset{\sim}{\longrightarrow}\calB_{\rss}.
$$
Moreover the natural injection $\t_{\reg}\rightarrow\g_{\rss}$ induces an isomorphism
$$
\ocalT_{\!\reg}\simeq\calG_{\rss}.
$$
Therefore, the morphism $(q_\rss,p_\rss)$  is the diagonal morphism
$$
\calT_{\reg}\longrightarrow\calT_{\reg}\times_\car\calG_{\rss}=\calT_{\reg}\times_\car\ocalT_{\!\reg}.
$$
We thus have the following cartesian diagram
$$
\xymatrix{\calT_{\reg}\ar[d]_{(q_\rss,p_\rss)}\ar[r]^{\pi_\reg}&\ocalT_{\!\reg}\ar[d]^{\Delta_{\ocalT_{\!\reg}}}\\
\calT_{\reg}\times_\car\ocalT_{\!\reg}\ar[r]^{\pi_{\reg}\times 1}&\ocalT_{\!\reg}\times_\car\ocalT_{\!\reg}}
$$
where the horizontal arrows are the quotient maps by $W$ and the vertical ones are the diagonal morphisms (which are $T$-torsors). Therefore
$$
\calN_\rss=(q_\rss,p_\rss)_!\Q\simeq (\pi_\reg\times 1)^*\Delta_{\ocalT_{\!\reg}!}\Q
$$
i.e. $\bN_{\rss}=\Delta_{\ocalT_{\!\reg}!}\Q$. The functors $\I_{\rss}:\dcb(\ocalT_{\!\reg})\rightarrow\dcb(\calG_{\rss})\simeq\dcb(\ocalT_{\!\reg})$ and $\R_{\rss}:\dcb(\ocalT_{\!\reg})\simeq\dcb(\calG_{\rss})\rightarrow\dcb(\ocalT_{\!\reg})$  are thus the identity functors. 

\subsection{The isomorphism $\R_\lambda\circ\I_\lambda\simeq 1$}\label{counit-section}

By (\ref{form2}) we can regard $\I_\lambda$ as a restriction functor. We thus consider the composition of correspondences

$$
\begin{footnotesize}
\xymatrix{&&\ocalT_\lambda\times_{\car_\lambda}\ocalT_\lambda\ar@/_3pc/[dddll]_b\ar@/^3pc/[dddrr]^a&&\\
&&\ocalT_\lambda\times_{\car_\lambda}\calG_\lambda\times_{\car_\lambda}\ocalT_\lambda\ar[u]_{\pr_{\ocalT_\lambda,\ocalT_\lambda}}\ar[rd]^{\pr_{23}}\ar[ld]_{\pr_{12}}&&\\
&\ocalT_\lambda\times_{\car_\lambda}\calG_\lambda\ar[rd]\ar[ld]&&\calG_\lambda\times_{\car_\lambda}\ocalT_\lambda\ar[rd]\ar[ld]&\\
\ocalT_\lambda&&\calG_\lambda&&\ocalT_\lambda}
\end{footnotesize}
$$
where the arrows are the obvious projections.
\bigskip

We put
$$
\bN_\lambda\circ_\calG\bN_\lambda:=\pr_{\ocalT_\lambda,\ocalT_\lambda !}(\pr_{12}^*\bN_\lambda\otimes\pr_{23}^*\bN_\lambda)=\pr_{\ocalT_\lambda,\ocalT_\lambda!}(\bN_\lambda\boxtimes_{\calG_\lambda}\bN_\lambda).
$$

Then it follows from \S\ref{cor} that the functor $\R_\lambda\circ\I_\lambda$ is the restriction functor of the cohomological correspondence
$$
(\ocalT_\lambda\times_{\car_\lambda}\ocalT_\lambda,\bN_\lambda\circ_\calG\bN_\lambda,a,b).
$$
\bigskip

\begin{theorem}
There is a natural isomorphism
$$
\bN_\lambda\circ_\calG\bN_\lambda\overset{\simeq}{\longrightarrow}\Delta_{\ocalT_\lambda!}\Q,
$$
where
$$
\Delta_{\ocalT_\lambda}:\ocalT_\lambda\rightarrow\ocalT_\lambda\times_{\car_\lambda}\ocalT_\lambda
$$
is the diagonal morphism. In particular we have $\R_\lambda\circ\I_\lambda\simeq 1$.
\label{maintheo1}\end{theorem}

As we did in \S \ref{main-descent}, we explain the proof in the nilpotent case which is the essential case. We thus consider the first diagram of \S \ref{main-descent} and we put

$$
\calN_\nil:=(q_\nil,p_\nil)_!\Q \in\dcb(\B(T)\times\calG_\nil)
$$

We have
$$
(\bN_\nil\circ_\calG\bN_\nil)|_{\B(T)\times\B(T)}=\pr_{\B(T),\B(T)\,!}(\calN_\nil\boxtimes_{\calG_\nil}\calN_\nil)
$$
where
$$
\pr_{\B(T),\B(T)}:\hat{\calS}_\nil\times_\calG\hat{\calS}_\nil=\B(T)\times\calG_\nil\times\B(T)\longrightarrow\B(T)\times\B(T)
$$
is the obvious projection.

Consider the decomposition of $(q_\nil,q_\nil)$ as
$$
\xymatrix{\calZ_\nil=\calB_\nil\times_{\calG_\nil}\calB_\nil\ar[r]^-{\delta_\nil\times\delta_\nil}&\hat{\calB}_\nil\times_{\calG_\nil}\hat{\calB}_\nil\ar[r]^-{\hat{f}_\nil\times\hat{f}_\nil}&\hat{\calS}_\nil\times_{\calG_\nil}\hat{\calS}_\nil\ar[r]^-{\pr_{\B(T),\B(T)}}&\B(T)\times\B(T)}
$$
By K\"unneth formula we have
\begin{equation}
(\bN_\nil\circ_\calG\bN_\nil)|_{\B(T)\times\B(T)}=(q_\nil,q_\nil)_!\Q.
\label{qq}\end{equation}
\bigskip

\noindent We consider the following commutative diagram
\begin{equation}
\xymatrix{\calB_\nil\ar[rr]^{\Delta_{\calB_\nil}}\ar[d]_{q_\nil}&&\calZ_\nil\ar[d]^{(q_\nil,q_\nil)}\\
\B(T)\ar[rr]^-{\Delta_{\B(T)}}&&\B(T)\times\B(T)}
\label{diagW}\end{equation}
with diagonal morphisms $\Delta_{\calB_\nil}$ and $\Delta_{\B(T)}$.

The morphism $\Delta_{\calB_\nil}:\calB_\nil\rightarrow\calZ_\nil$ is schematic closed and so by adjunction we have a morphism
$$
\Q\rightarrow\Delta_{\calB_\nil*}\Q=\Delta_{\calB_\nil!}\Q,
$$
and so by applying the functor $(q_\nil,q_\nil)_!$ we get a morphism
\begin{equation}
\xymatrix{(q_\nil,q_\nil)_!\Q\ar[r]&\Delta_{\B(T) !}q_{\nil\,!}\Q.}
\label{eq11}\end{equation}
On the other hand the morphism $q_\nil$ being smooth with fibers isomorphic to $[\u/U]$, we have by adjunction a morphism
\begin{equation}
\xymatrix{q_{\nil\,!}\Q=q_{\nil\,!}(q_\nil)^!\Q\ar[r]&\Q}
\label{q!}\end{equation}
which is an isomorphism and so we obtain an isomorphism
\begin{equation}
\xymatrix{\Delta_{\B(T) !}q_{\nil\, !}\Q\ar[r]&\Delta_{\B(T) !}\Q.}
\label{eq12}\end{equation}
Composing (\ref{eq11}) with (\ref{eq12}), we end up with a morphism
\begin{equation}
\xymatrix{(q_\nil,q_\nil)_!\Q\ar[r]&\Delta_{\B(T) !}\Q.}
\label{Mor1}\end{equation}

From (\ref{qq}) together with (\ref{Mor1}) we get a morphism
$$
\bN_\nil\circ_\calG\bN_\nil\longrightarrow(\pi_\nil\times\pi_\nil)_!\Delta_{\B(T) !}\Q=\Delta_{\B(N) !}\pi_{\nil\,!}\Q
$$
where $\pi_\nil$ is the canonical map $\B(T)\rightarrow\B(N)$.
\bigskip

Composing with the adjunction morphism
$$
\pi_{\nil\,!}(\pi_\nil)^*=\pi_{\nil\,!}(\pi_\nil)^!\longrightarrow 1
$$
we find
\begin{equation}
\xymatrix{\bN_\nil\circ_\calG\bN_\nil\ar[r]&\Delta_{\B(N) !}\Q}
\label{co-unit}\end{equation}
\bigskip

Consider the cartesian diagram
$$
\xymatrix{ W\times\B(T)\ar[rr]^f\ar[d]^{\pi_\nil}&&\B(T)\times\B(T)\ar[d]^{\pi_\nil\times\pi_\nil}\\
\B(N)\ar[rr]^-{\Delta_{\B(N)}}&&\B(N)\times\B(N)}
$$
where the top arrow is given by $(w,t)\mapsto (w(t),t)$ and the bottom arrow is the diagonal morphism.

If we denote by $\Delta_{\B(T),w}:\B(T)\rightarrow\B(T)\times\B(T)$ the $w$-twisted diagonal morphism induced by $t\mapsto (w(t),t)$, then 	
$$
f_!\Q=\bigoplus_{w\in W}\Delta_{\B(T),w\, !}\Q=(\pi_\nil\times\pi_\nil)^*\Delta_{\B(N)!}\Q
$$
and so the pullback of (\ref{co-unit}) along the map $\pi_\nil\times\pi_\nil$ provides a morphism
\begin{equation}
(q_\nil,q_\nil)_!\Q\rightarrow\bigoplus_{w\in W}\Delta_{\B(T),w\,!}\Q.
\label{Mor2}\end{equation}
of $W\times W$-equivariant  complexes.

Theorem \ref{maintheo1} is thus a consequence of the following result.

\begin{theorem}
The morphism (\ref{Mor2}), and so the morphism (\ref{co-unit}), is an isomorphism.
\label{theoMor2}\end{theorem}

It is enough to show that the morphism
$$
\calH^i(q_\nil,q_\nil)_!\Q\rightarrow \calH^i\left(\bigoplus_{w\in W}\Delta_{\B(T),w\,!}\Q\right).
$$
induced by the morphism (\ref{Mor2}) is an isomorphism for all $i$. 
\bigskip

To prove this we  prove that we have an isomorphism

$$
\calH^i(q,q)_!\Q\rightarrow \calH^i\left(\bigoplus_{w\in W}\Delta_{\calT,w\,!}\Q\right)
$$
where $\Delta_{\calT,w}:\calT\rightarrow\calT\times_\car\calT$ is induced by $t\mapsto (w(t),t)$.

\bigskip

\begin{proposition}
The sheaf $\calH^i(q,q)_!\Q$ is  (up to a shift by $-n$) perverse  and is the intermediate extension of its restriction to the open substack of regular elements of $\calT\times_\car\calT$. In particular, it is  $W\times W$-equivariant.\label{IE1}\end{proposition}

To prove the proposition we will use a weight argument. Therefore we choose a finite subfield $k_o$ of $k$ such that $G, T,$ and $B$ are defined over $k_o$ and $T$ is split over $k_o$.
\bigskip

Recall that we have a stratification (see \S \ref{SS})

$$
\calZ=\coprod_{w\in W}\calZ_w.
$$

The restriction of $(q,q)$ to $\calZ_w$ is the composition of the projection $q_w:\calZ_w\rightarrow\calT$ followed by the morphism $
\Delta_{\calT,w}:\calT\rightarrow\calT\times_\car\calT$.
\bigskip

As the cohomology of $[\u_w/U_w]$ is trivial we deduce that
\begin{equation}
((q,q)|_{\calZ_w})_!\Q=\Delta_{\calT,w\,!}\Q.
\label{Mor3}\end{equation}

We choose a total order $\{w_0,w_1,\dots\}$ on $W$ so that we have a decreasing filtration of closed substacks
$$
\calZ=\calZ_0\supset\calZ_1\supset\cdots\supset\calZ_{|W|-1}\supset\calZ_{|W|}=\emptyset
$$
satisfying $\calZ_i\backslash\calZ_{i+1}=\calZ_{w_i}$ for all $i$.
\bigskip

By \cite[Chapitre 6, 2.5]{SGA} we have a spectral sequence
$$
E_1^{ij}=\calH^{i+j}\left((q,q)|_{\calZ_i\backslash\calZ_{i+1}}\right)_!\Q\Rightarrow \calH^{i+j}(q,q)_!\Q.
$$
\bigskip

\begin{lemma}
This spectral sequence degenerates at $E_1$.
\label{specseq}\end{lemma}
\bigskip

\begin{proof}
As $\calZ_i\backslash\calZ_{i+1}=\calZ_{w_i}$ and $\calT=T\times\B(T)$, by (\ref{Mor3}) we have
\begin{equation}
E_1^{ij}=\calH^{i+j}\Delta_{\calT,w_i\,!}\Q=\bigwedge^{2n-i-j}X_*(T)\otimes\left(\Delta_{\t,w_i\, !}\Q\boxtimes {\Q}_{,\B(T)\times\B(T)}\right)(n-i-j)
\label{PE}\end{equation}
where $\Delta_{\t,w}:\t\rightarrow \t\times_\car \t$ is given by $t\mapsto (w(t),t)$. Notice that 
$$
\Delta_{\t,w_i\,!}\Q\boxtimes {\Q}_{,\B(T)\times\B(T)}
$$
is (up to a shift) perverse and pure of weight $0$. Hence  $E_1^{ij}$ is (up to a shift) pure of weight $2(i+j)$.
\bigskip

Therefore the sources and the targets of the differentials $d_1$ are (up to a shift) perverse of different weight and so $d_1$ vanishes. We keep applying the same argument for the other differentials and we deduce that the spectral sequence degenerates at $E_1$, i.e. $E_1=E_\infty$.
\end{proof}
\bigskip

\begin{proof}
[Proof of Proposition \ref{IE1}]
It follows from Lemma \ref{specseq} that $\calH^k(q,q)_!\Q$ is a successive extension of $E_1^{ij}$'s which are intermediate extension of their restriction to the open substack of regular elements by (\ref{PE}). It is thus perverse (up to a shift by $-n$) and so the proposition is a consequence of Proposition \ref{IE}.
\end{proof}
\bigskip

Following the proof of (\ref{eq11}) we have  also a  morphism

$$
(q,q)_!\Q\longrightarrow \Delta_{\calT\, !}\Q
$$
where $\Delta_\calT:\calT\rightarrow\calT\times_\car\calT$ is the diagonal morphism. We thus get a morphism

$$
\calH^i(q,q)_!\Q\longrightarrow\calH^i\Delta_{\calT\, !}\Q
$$
and since, by Proposition \ref{IE1} the sheaf $\calH^i(q,q)_!\Q$ is $W\times W$-equivariant, we get a morphism

$$
\calH^i(q,q)_!\Q\longrightarrow\bigoplus_{w\in W}\calH^i\Delta_{\calT,w\, !}\Q.
$$
The source and the target of this morphism are intermediate extension of their restriction to regular elements (see Proposition \ref{IE1}). It is thus an isomorphism as its restriction to regular element is an isomorphism.

\section{Main results}

\subsection{Derived categories}

Fix $\lambda\in\frak{L}$. In this section we choose a geometric point $c:\mathrm{Spec}(k)\rightarrow\car_\lambda$. For any stack $\calX$ above $\car_\lambda$ we put
$$
\calX_c:=\mathrm{Spec}(k)\times_\car\calX,
$$
and for any complex $K\in\dcb(\calX)$ we denote by $K_c\in\dcb(\calX_c)$ the pullback of $K$ along the base change morphism $\calX_c\rightarrow\calX$.

We have a commutative diagram with Cartesian squares
\begin{equation}
\xymatrix{\overline{\calT_c}\ar[d]_{\iota}&\overline{\calT_c}\times\calG_c\ar[r]^-{\pr_2}\ar[d]^\iota\ar[l]_-{\pr_1}&\calG_c\ar[d]^\iota\\
\ocalT_\lambda&\ocalT_\lambda\times_{\car_\lambda}\calG_\lambda\ar[r]\ar[l]&\calG_\lambda}
\label{c}\end{equation}
We then define the pair $(\R_c,\I_c)$ of adjoint functors
$$
\xymatrix{\dcb(\overline{\calT_c})\ar@/_/[rr]_-{\I_c}&&\dcb(\calG_c)\ar@/_/[ll]_-{\R_c}}
$$
by
$$
\I_c: K\mapsto\pr_{2!}(\bN_c\otimes\pr_1^*(K)),\hspace{1cm}\R_c : K\mapsto\pr_{1!}(\bN_c\otimes\pr_2^*(K)).
$$
The following lemma follows from the base change theorem and the projection formula.

\begin{lemma} We have the commutation formulas.

\noindent (a)
$$
\iota_!\circ\I_c=\I_\lambda\circ\iota_!,\hspace{1cm}\I_c\circ\iota^*=\iota^*\circ\I_\lambda.
$$
(b)
$$
\iota_!\circ\R_c=\R_\lambda\circ\iota_!,\hspace{1cm}\R_c\circ\iota^*=\iota^*\circ\R_\lambda.
$$
\label{commfor}\end{lemma}

\begin{theorem}
We have $$
\R_c\circ\I_c\simeq 1.\hspace{1cm}
$$

\label{513}\end{theorem}

\begin{proof}
The isomorphism of functors $\R_c\circ\I_c\simeq 1$ is obtained by taking the pullback along $\overline{\calT_c}\times\overline{\calT_c}\rightarrow\ocalT_\lambda\times_{\car_\lambda}\ocalT_\lambda$ of the isomorphism of complexes
$$
\bN_\lambda\circ_\calG\bN_\lambda\overset{\simeq}{\longrightarrow}\Delta_{\ocalT_\lambda\, !}\Q
$$
of Theorem \ref{maintheo1}.
\end{proof}

We have also the functor 
$$
\Ind_c:\dcb(\calT_c)\longrightarrow\dcb(\calG_c)
$$
defined by base change from the diagram (\ref{Indbis}). It satisfies

$$
\Ind_c=\I_c\circ\pi_{c\, !}
$$
where $\pi_c$ is the $W$-torsor $\calT_c\rightarrow\ocalT_c$. When $c=0$ we will use the notation $\pi_\nil$ instead of $\pi_c$. 
\bigskip

We define $\dcb(\calG_c)^\Sp$ as the triangulated subcategory of $\dcb(\calG_c)$ generated by the simple direct factors of $\Ind_c(\Q)$ which, by Springer theory, are indexed by the irreducible characters of the Weyl group $W_c$ of the centralizer of $c$. 

\begin{proposition}The functor $\I_c:\dcb(\ocalT_c)\rightarrow \dcb(\calG_c)^\Sp$ induces a bijection between simple perverse sheaves compatible with parametrization by irreducible representations of $W_c$.
\label{Springer}\end{proposition}

\begin{proof}The essential case is $c=0$ in which case we have $\calT_c=\B(T)$, $\ocalT_c=\B(N)$ and $\dcb(\calG_c)^\Sp=\dcb(\calG_\nil)^\Sp$. The simple objects of $\dcb(\B(N))$ and $\dcb(\calG_\nil)^\Sp$ are both parametrized by the irreducible characters of $W$. This is  Springer theory in the second case and in the first case the simple objects are the simple direct  summands of $\pi_{\nil\,*}\Q$. The pullback functor $\overline{s}_\nil^![n](n)$ along the obvious morphism $\overline{s}_\nil:\B(N)\rightarrow\B(W)$ induces an equivalence between the categories of perverse sheaves that is compatible with the indexation of the simple perverse sheaves by the irreducible characters of $W$. By Lemma \ref{compatible} applied to nilpotents, we have

$$
{^p}\I_\nil=\I_\nil\circ\overline{s}_\nil^![n](n).
$$
We are thus reduced to prove the statement of the proposition for ${^p}\I_\nil:\calM(\B(W))\rightarrow\calM(\calG_\nil)^{\rm Spr}$. Unlike $\I_\nil$, the construction of ${^p}\I_\nil$ extends to a global functor ${^p}\I$. A standard argument using the functor ${^p}\I$ and its restriction to semisimple regular elements proves the proposition for ${^p}\I_\nil$.  
\end{proof}

\begin{proposition}The functor $\I_c$ induces an equivalence of categories $\dcb(\ocalT_c)\rightarrow\dcb(\calG_c)^\Sp$ with inverse functor $\R_c$.
\end{proposition}

\begin{proof}Since $\R_c\circ \I_c\simeq 1$, the functor $\I_c$ induces an injection between the ${\rm Ext}^i$ of  simple perverse sheaves. On the other hand, by the proposition below  together with \cite[Theorem 4.6] {Achar} and Proposition \ref{Springer}, we know that the ${\rm Ext}^i$ of the simple objects that corresponds under $\I_c$ have the same dimension and so $\I_c$ induces an isomorphism between the ${\rm Ext}^i$. By a lemma of Beilinson \cite[Lemma 6]{Schnurer} , the functor $\I_c$ induces thus an equivalence of categories $\dcb(\ocalT_c)\simeq\dcb(\calG_c)^\Sp$.
\end{proof}

Consider the decomposition
$$
\pi_{\nil\, !}\Q=\bigoplus_\varphi V_\varphi\otimes \mathcal{L}_\varphi
$$
where the sum runs over the irreducible $\Q$-characters of $W$, $V_\varphi$ is an irreducible representation affording the character $\varphi$ and the $\mathcal{L}_\varphi$ are the  irreducible smooth $\ell$-adic sheaf on $\B(N)$.

\begin{proposition}For any irreducible characters $\varphi$ and $\varphi'$ of $W$ we have

$$
{\rm Ext}^i(\mathcal{L}_\varphi,\mathcal{L}_{\varphi'})=\left(V_\varphi\otimes H^i(\B(T),\Q)\otimes V_{\varphi'}^*\right)^W.
$$
\end{proposition}

\begin{proof} The proof is similar  to that of \cite[Theorem 4.6]{Achar}. We give it for the convenience of the reader. We have

\begin{align*}
{\rm Ext}^i(\mathcal{L}_\varphi,\mathcal{L}_{\varphi'})&={\rm Hom}_W\left(V_\varphi^*,{\rm Ext}^i(\pi_{\nil\,!}\Q,\mathcal{L}_{\varphi'})\right)\\
&={\rm Hom}_W\left(V_\varphi^*,{\rm Ext}^i(\Q,\pi_\nil^*(\mathcal{L}_{\varphi'}))\right)\\
&={\rm Hom}_W\left(V_\varphi^*,{\rm Ext}^i_{\B(T)}(\Q,\Q)\otimes{\rm Hom}(\Q,\pi_\nil^*(\mathcal{L}_{\varphi'}))\right)\\
&={\rm Hom}_W\left(V_\varphi^*,H^i(\B(T),\Q)\otimes{\rm Hom}(\pi_{\nil\,!}\Q,\mathcal{L}_{\varphi'})\right)\\
&=\left(V_\varphi\otimes H^i(\B(T),\Q)\otimes V_{\varphi'}^*\right)^W.
\end{align*}
The third equality follows from the canonical isomorphism

$$
\pi_\nil^*(\mathcal{L}_{\varphi'})\simeq \Q\otimes{\rm Hom}(\Q,\pi_\nil^*(\mathcal{L}_{\varphi'}))
$$
as the category $\calM(\B(T))$ is semisimple linear with unique simple object $\Q[-n]$.

\end{proof}

Define $\dcb(\calG_\lambda)^\Sp$ as the full subcategory of $\dcb(\calG_\lambda)$ of complexes $K$ such that $K_c\in\dcb(\calG_c)^\Sp$ for all geometric point $c$ of $\car_\lambda$.
\bigskip

\begin{remark}If $G$ is of type $A$ with connected center then we have $\dcb(\calG_\lambda)^\Sp=\dcb(\calG_\lambda)$ for all $\lambda$.
\label{rem-lambda}\end{remark}
\bigskip

\begin{theorem}The functor $\I_\lambda$ induces an equivalence of categories $\dcb(\ocalT_\lambda)\rightarrow\dcb(\calG_\lambda)^\Sp$ with inverse functor $\R_\lambda$. 
\label{adj2-lambda}\end{theorem}

\begin{proof}Since $\R_\lambda\circ\I_\lambda\simeq 1$, it suffices to see that the morphism

$$
K\longrightarrow \I_\lambda\circ\R_\lambda(K)
$$
is an isomorphism for all $K\in\dcb(\calG_\lambda)^\Sp$. But it follows from the fact that it is true over $\calG_c$ for any geometrical point $c$ of $\car_\lambda$ by the above proposition.
\end{proof}

\subsection{Perverse sheaves}\label{resper}

Let $({^p}\R,{^p}\I)$ be the pair of adjoint functors defined in \S \ref{perInd}.

\begin{theorem} The adjunction map

$$
{^p}\R\circ{^p}\I\rightarrow 1
$$
is an isomorphism.
\label{adj1}\end{theorem}

It is enough to prove that for any $\lambda\in\frak{L}$ and any $K\in\calM(\overline{\t})$  the restriction

$$
\left({^p}\R\circ{^p}\I(K)\right)|_{\ocalT_\lambda}\rightarrow K|_{\ocalT_\lambda}
$$
of the adjunction morphism ${^p}\R\circ{^p}\I(K)\rightarrow K$ along the map $\ocalT_\lambda\rightarrow\overline{\t}$ is an isomorphism.
\bigskip

Therefore the theorem is a consequence of Theorem \ref{maintheo1} together with the following result which is straightforward.

\begin{proposition}For any $K\in\calM(\overline{\t})$ and $L\in\calM(\calG)$ we have

$$
{^p}\I(K)|_{\calG_\lambda}=\I_\lambda(K|_{\ocalT_\lambda})[-n],\hspace{1cm}{^p}\R(L)|_{\ocalT_\lambda}=\R_\lambda(L|_{\calG_\lambda})[n].
$$
\label{propres}\end{proposition}

Denote by $\calM(\calG)^\Sp$ the full subcategory if $\calM(\calG)$ of perverse sheaves $K$ such that $K_c\in\dcb(\calG_c)^\Sp$ for all geometric point $c$ of $\car$.
\bigskip

\begin{theorem} The functor ${^p}\I$ induces an equivalence of catgeories $\calM(\overline{\t})\rightarrow\calM(\calG)^\Sp$ with inverse functor ${^p}\R$. 

In particular if $G$ is of type $A$ with connected center then $\calM(\calG)^\Sp=\calM(\calG)$ and so  the categories $\calM(\overline{\t})$ and $\calM(\calG)$ are equivalent.
\label{maintheo-perverse}\end{theorem}

\begin{proof}Thanks to Theorem \ref{adj1}, it remains to prove that

$$
K\rightarrow {^p}\I\circ{^p}\R(K)
$$
is an isomorphism for all $K\in\calM(\calG)^\Sp$. We are thus reduced to prove  that this is an isomorphism after restriction to $\calG_\lambda$ for all $\lambda\in\frak{L}$. But this is a consequence of Proposition \ref{propres} and Theorem \ref{adj2-lambda}.
\end{proof} 

\bigskip

\begin{large}{\bf Conflicts of interest:} none.\end{large}

\end{document}